\documentclass[a4paper,reqno]{amsart}
\usepackage{amsfonts}
\usepackage{latexsym}
\usepackage{amsmath,amsthm,amssymb, color, marginnote}
\usepackage[all]{xy}
  \usepackage{mathrsfs}
\usepackage{MnSymbol}
\usepackage{marginnote}
\usepackage{hyperref}
\hypersetup{colorlinks, linkcolor=blue,citecolor=blue}

\usepackage{graphicx,psfrag,epsfig}

\theoremstyle{plain}
\newtheorem{theorem}{Theorem}

\newtheorem{lemma}[theorem]{Lemma}
\newtheorem{corollary}[theorem]{Corollary}
\newtheorem{proposition}[theorem]{Proposition}
\theoremstyle{definition}
\newtheorem{definition}[theorem]{Definition}
\newtheorem{remark}[theorem]{Remark}

\def\bdef{\begin{definition}}
\def\endef{\end{definition}}
\def\bthm{\begin{theorem}}
\def\ethm{\end{theorem}}
\def\blm{\begin{lemma}}
\def\elm{\end{lemma}}
\def\brm{\begin{remark}}
\def\erm{\end{remark}}
\def\bprop{\begin{proposition}}
\def\eprop{\end{proposition}}
\def\bcor{\begin{corollary}}
\def\ecor{\end{corollary}}
\def\be{\begin{eqnarray}}
\def\ee{\end{eqnarray}}
\def\beal{\begin{aligned}}
\def\enal{\end{aligned}}

\def\om{\omega}

\def\eps{\varepsilon}
\def\phi{\varphi}

\def\R{\mathbb R}
\def\C{\mathbb C}
\def\Nn{\mathbb N}

\def\Z{\mathbb Z}

\def\B{\mathcal B}
\def\Pe{\mathcal P}

\def\cF{\mathcal F}
\def\~{\tilde}

\def\g{\gamma}

\def\cK{\mathcal K}
\def\cH{\mathcal H}

\def\cD{\mathcal D}

\def\PP{\mathbb{P}}
\def\EE{\mathbb{E}}

\def\Co{|_{C^0}}
\def\Cod{|_{C^1}}
\def\Cdv{|_{C^2}}
\def\Cm{|_{C^m}}
\def\vs{v(s,x)}
\def\kam{\kappa}
\def\p{\partial}
 \newcommand{\strela}{\rightharpoonup}

\def\lan{\langle}
\def\ran{\rangle}

\def\be{\begin{equation}}
\def\ee{\end{equation}}
\def\bdef{\begin{definition}}
\def\endef{\end{definition}}
\def\blm{\begin{lemma}}
\def\elm{\end{lemma}}
\def\beal{\begin{aligned}}
\def\enal{\end{aligned}}

\newtheorem*{Pf}{Proof}
\renewenvironment{proof}{\begin{Pf} \begin{upshape}} {\end{upshape} \qed\end{Pf}}

\title[Energy transfer in NLS]{On the energy transfer to high frequencies  in  the  damped/driven nonlinear Schr\"odinger  equation (extended version)}

\begin{document}
\author{Guan Huang}
\address{Yau Mathematical Sciences Center, Tsinghua University, Beijing, China}
\email{huangguan@tsinghua.edu.cn}
\author{Sergei Kuksin}
\address{Institut de Math\'emathiques de Jussieu--Paris Rive Gauche, CNRS, Universit\'e Paris Diderot, UMR 7586, Sorbonne Paris Cit\'e, F-75013, Paris, France \& School of Mathematics, Shandong University, Jinan, Shandong,  China \& Saint Petersburg State University, Universitetskaya nab., St. Petersburg, Russia}
\email{Sergei.Kuksin@imj-prg.fr}
\maketitle
\numberwithin{equation}{section}
\begin{abstract}
We consider a damped/driven  nonlinear Schr\"odinger  equation in an $n$-cube $K^{n}\subset\mathbb{R}^n$, $n$ is arbitrary, 
under Dirichlet boundary conditions
\[
u_t-\nu\Delta u+i|u|^2u=\sqrt{\nu}\eta(t,x),\quad x\in K^{n},\quad u|_{\partial K^{n}}=0, \quad \nu>0,
\] where $\eta(t,x)$ is a random force that is white in time and smooth in space.  It
is known that the Sobolev norms of solutions satisfy
$
\| u(t)\|_m^2 \le C\nu^{-m},
$
uniformly in $t\ge0$ and  $\nu>0$. In this work we  prove that for small  $\nu>0$ and  any initial data, with large probability  the Sobolev norms $\|u(t,\cdot)\|_m$ of the solutions with $m>2$  become large at least to the order of $\nu^{-\kappa_{n,m}}$ with $\kappa_{n,m}>0$, on  time intervals of order $\mathcal{O}(\frac{1}{\nu})$. \end{abstract}

\section{Introduction}
In this work we study a damped/driven nonlinear Schr\"odinger equation 
\begin{equation}\label{m-eq}
u_t-\nu\Delta u+i|u|^2u=\sqrt{\nu}\eta(t,x),\quad x\in K^{n},
\end{equation}
where $n$ is any, 
$0<\nu \le1$ is the viscosity constant and the random force $\eta$ is white in time $t$ and regular in $x$. The equation is considered under the 
 odd periodic boundary conditions,
\[u(t,\dots,x_j,\dots)=u(t,\dots,x_j+2\pi,\dots)=-u(t,\dots,x_j+\pi,\dots),\quad j=1,\dots,n.
\]
The latter implies that $u$ vanishes on the boundary of the cube of half-periods $K^{n}= [0, \pi]^n$,
$$
u\mid_{\partial K^{n}} =0. 
$$
We denote by $\{\varphi_d(\cdot),\; d=(d_1,\dots,d_n)\in\mathbb{N}^n\}$ the  trigonometric basis in the space of odd periodic functions,
$$
\varphi_d(x)=(\tfrac{2}{\pi})^{\frac{n}{2}}\sin(d_1x_1)\cdots\sin(d_nx_n).
$$
The basis  is orthonormal with respect to the scalar product $\llangle \cdot,\cdot \rrangle$ in 
$L_2(K^{n}, \pi^{-n} dx)$, 
$$
 \llangle u,v\rrangle =  \int_{K^n} \lan u(x),  v(x)\ran \pi^{-n} dx,
$$
where $\lan \cdot, \cdot\ran$ is the real scalar product in $\C$, $\lan u, v \ran =\Re u\bar v$. It is formed by 
eigenfunctions of the Laplacian:
$$
(-\Delta)\varphi_d=|d|^2\varphi_d.  
$$

 The force $\eta(t,x)$ is a random field of the form 
\[\eta(t,x)=\frac{\partial}{\partial t}\xi(t,x),\quad \xi(t,x)=\sum_{d\in\mathbb{N}^n}b_d\beta_d(t)\varphi_d(x).\]
Here $\beta_d(t)=\beta^R_d(t)+i\beta^I_d(t)$, where $\beta_d^R(t)$, $\beta_d^I(t)$ are independent real-valued standard Brownian motions, defined on a complete probability space $(\Omega,\mathcal{F},\mathbb{P})$ with a filtration  $\{\mathcal{F}_t;t\geqslant0\}$. 
The set of real numbers $\{b_d,\;d\in\mathbb{N}^n\}$ is assumed to form a non-zero sequence, satisfying 
\begin{equation}\label{B_m}
0<B_{m_*}<\infty, \quad m_*= \min\{ m\in\Z: m>n/2\},
\end{equation}
where for $k\in \R$ we set
$$
B_k:=\sum_{d\in\mathbb{N}^n}|d|^{2k}  |b_d|^2\le \infty.
 $$
 For $m\ge 0$ we  denote by $H^m$ the Sobolev space of order $m$, formed by complex odd periodic functions, equipped with the homogeneous norm,
\[\|u\|_m=\|(-\Delta)^{\frac{m}{2}}u\|_0,\]
where $\|\cdot\|_0$ is the $L^2$-norm on  $K^{n}$, 
$
\|u\|^2_0 =  \llangle u,u\rrangle.
$
 If we write $u\in H^m$ as Fourier series, 
$u(x)=\sum_{d\in\mathbb{N}^n}u_d\varphi_d(x),$ then 
$
\|u\|_m^2=\sum_{d\in\mathbb{N}^n}|d|^{2m}| u_d|^2.
$

Equation \eqref{m-eq} with small $\nu>0$ is a natural model for the small-viscosity Navier--Stokes system with a random 
force in dimensions 2 and 3, describing  2d and 3d turbulence. Indeed, the former is obtained from the latter by replacing the 
Euler equation 
$\ 
u_t + (u\cdot \nabla) u +\nabla p=0, \; \text{div}\; u=0, 
$
which is a 2--homogeneous Hamiltonian PDE, with  the 3-homogeneous Hamiltonian system
$
u_t +i |u|^2u=0.
$
See more in \cite{LO, k-gafa1999}. So a progress in the study of eq.~\eqref{m-eq} with small $\nu$ should help to understand turbulence a bit better. 

\smallskip

The global solvability of eq.~\eqref{m-eq} for any space dimension $n$ 
is established in~\cite{kuk1999, Kuk-N2013}. It is proved there  that if 
\begin{equation}\label{ic1} 
u(0,x)=u_0(x), 
\end{equation}
where $u_0\in H^m$, $m>\frac{n}{2}$, and if $B_m<\infty$, then the problem \eqref{m-eq}, \eqref{ic1} has a unique strong solution $u(t,x)$  in $H^m$
which we write as $u(t,x;u_0)$ or $u(t;u_0)$. Its norm satisfies 
$$
\EE \|u(t;u_0)\|_m^2 \le C_m \nu^{-m},\quad t\ge0, 
$$
where $C_m$ depends on $\|u_0\|_m, |u_0|_\infty$ and $B_m, B_{m_*}$. 
  Furthermore, denoting by $C_0(K^{n})$ the space of continuous complex functions on $K^{n}$, vanishing at $\partial K^{n}$, we have that 
the solutions $u(t,x)$ 
  define a Markov process in the space $C_0(K^{n})$. 
  Moreover, if the noise $\eta(t,\cdot)$ is sufficiently non-degenerate (depending on $\nu$, see Theorem \ref{mix-measure}), then this  process is mixing.

Our goal is to  study the growth of higher Sobolev norms for solutions of the equation~\eqref{m-eq} as $\nu\to0$ on 
 time intervals of order $\mathcal{O}(\frac{1}{\nu})$.  
The main result of this work is the following. 

 \begin{theorem}\label{m-theorem}
For any real number $m>2$, in 	addition to \eqref{B_m},  assume that $B_m<\infty$. Then 
  there exists $\kappa_{n,m}>0$ such that for   every  fixed quadruple  $(\delta,\kappa,\mathscr{K}, 
 T_0)$, where
$$
 \kappa\in(0,\kappa_{n,m}),\quad \delta\in(0,\tfrac{1}{8}), \quad 
  \mathscr{K},   T_0>0,
$$
there exists $\nu_0>0$ with the property that if $0<\nu\le \nu_0$, then for every $u_0\in H^m \cap C_0(K^{n})$, satisfying 
\be\label{initial}
 |u_0|_\infty\leqslant \mathscr{K}, \quad \| u_0\|_m \le \nu^{-\kappa m},
 \ee
 the solution $u(t,x;u_0)$ is such that 
\begin{enumerate}
\item 
\[
\mathbb{P}\big\{ \sup_{t \in [t_0, t_0+ T_0\nu^{-1}]} 
\|u_\nu^\omega(t)\|_m> \nu^{-m\kappa }\big\} \ge 1-\delta, \quad \forall\, t_0\ge 0.
\]
\item If  $m$ is an integer, $m\ge3$, then a possible choice of $\kappa_{n,m}$ is $\kappa_{n,m}= \tfrac1{35}$, and 
 there exists  $C\ge1$, depending  on $\kappa< \tfrac1{35}$, 
  $\mathscr{K},  m,B_{m_*}$ and  $B_m$, such that 
\begin{equation}\label{avb}
C^{-1} \nu^{-2m\kappa +1}\leqslant\mathbb{E}\Big(\nu \int_{t_0}^{t_0+\nu^{-1}}\|u_\nu(s)\|_m^2ds\Big)\leqslant C \nu^{-m}, 
\quad \forall\, t_0\ge 0.
\end{equation}
\end{enumerate}
\end{theorem}

A similar result holds for the classical $C^k$-norms of  solutions:
 \begin{proposition}\label{m-theorem2}
For any integer  $m \ge2$  in 	addition to \eqref{B_m}  assume that $B_m<\infty$. Then  
for   every  fixed triplet $K, \cK, T_0>0$  and any $0<\kappa<1/16$ we have 
\be\label{conver}
\mathbb{P}\big\{ \sup_{t \in [t_0, t_0+ T_0\nu^{-1}]} 
|u_\nu^\omega(t;u_0)|_{C^m} >  K \nu^{-m\kappa }\big\} \to1 \quad\text{as} \quad \nu\to0,
\ee
for each $t_0\ge0$, if $u_0$ satisfies $|u_0|_\infty\le \cK$,   $|u_0|_{C^m} \le \nu^{-\kappa m}$. The rate of convergence depends
only on the triplet  and $\kappa$. 
\end{proposition}

The result is proved by adapting the argument from \cite[Section 5]{k-gafa1999} to the current settings
and arguing similar to when proving Theorem~\ref{m-theorem}; see in Appendix \ref{appen-cm}.
Due to \eqref{conver}, for any $m>2+n/2$ we have 
$$
\PP \Big\{ \sup_{T_0 \le t \le t_0+{T_0}{\nu^{-1}}}  \| u(t)\|_{m} \ge K \nu^{-  \lfloor m-\frac{n}{2} \rfloor  \kappa} 
\Big\} \to 1\quad \text{as} \quad \nu\to0,
 $$
 for every $K>0$ and  $0<\kappa <1/16$, where for  $a\in\R$ we denote
 $
 \lfloor a \rfloor = \max\{ n\in\Z: n<a\}. 
 $
 This result improves the first assertion of Theorem~\ref{m-theorem} for large $m$.

We have the following  two corollaries from Theorem~\ref{m-theorem}, 
 valid if the Markov process defined by the equation \eqref{m-eq} is mixing:

\begin{corollary}\label{c_3}
Assume that $B_m<\infty$ for all $m$ and $b_d\ne0$ for all $d$. Then eq.~\eqref{m-eq} is mixing and for any 
$\kappa<1/35$ and $0<\nu\le\nu_0$ its 
  unique stationary measure 
$\mu_\nu$ satisfies 
\be\label{bb1}
C^{-1} \nu^{-2m\kappa+1} \le \int \| u\|_m^2\mu_\nu(du) \le C\nu^{-m}, \quad 3\le m\in\Nn.
\ee
Here $C$ and $\nu_0$ are as in Theorem \ref{m-theorem}. 
\end{corollary}

\begin{corollary}\label{c_4}
Under the assumptions of Corollary \ref{c_3}, for any $u_0\in C^\infty$ we have 
$$
\tfrac12 C^{-1} \nu^{-2m\kappa+1} \le \EE \| u (s;u_0)\|_m^2 \le 2C\nu^{-m}, \quad 3\le m\in\Nn,
$$
if $s\ge T(\nu, u_0, \kappa, B_m, B_{m_*})$, where $C$ is the same as in  \eqref{bb1}. 
\end{corollary}
\begin{remark}
The exponent in the lower bound here can be slightly improved to $-(2m-7.001)\kappa$, see Remark \ref{stationary-lower}.
\end{remark}

Theorem \ref{m-theorem} rigorously establishes the energy cascade to high frequencies for solutions of eq. \eqref{m-eq} with small $\nu$. 
Indeed, if $u_0(x)$ and $\eta(t,x)$ are smooth functions of $x$ (or even trigonometric polynomials of $x$), then in view of 
\eqref{avb} for $0<\nu\ll1$ and $t \gtrsim \nu^{-1}$ a substantial part of the energy
$
\frac12 \sum |u_d(t)|^2
$
of a solution $u(t,x;u_0)$ is carried by high modes $u_d$, $|d|\gg1$. Relation \eqref{avb} (valid for all integer $m\ge3$) also means that the 
averaged in time space-scale $l_x$ of solutions for \eqref{m-eq} satisfies
$
l_x \in [\nu^{1/2}, \nu^{1/35}], 
$
and goes to zero with $\nu$ (see \cite{k-gafa1999} and \cite[Section 7.1]{BK}).  We recall that the energy cascade to high frequencies 
and formation of short space-scale is the driving force of the Kolmogorov theory of turbulence, see \cite{Fr} and
\cite[Section 6]{BK}. 

Two-sided estimates, similar to \eqref{avb}, are known for solutions of the stochastic Burgers equation on $S^1$:
\be\label{B}
 u_t -\nu u_{xx} +uu_x = \eta(t,x), \quad x\in S^1,\;\; \int u\,dx= \int\eta\,dx=0
\ee
(in view of the strong dissipativity in the small-viscosity Burgers equation, 
to get solutions of order one the force should be $\sim1$, 
rather then $\sim\sqrt\nu$ as in \eqref{m-eq}). Due to some remarkable features of the Burgers equation, the lower and upper estimates 
for Sobolev norms of solutions for \eqref{B} are asymptotically sharp in the sense that they contain $\nu$ in the same negative degree.
Moreover, they hold after averaging on time--intervals of order one; see \cite[Section 2]{BK}.
\medskip

Deterministic versions of the result of Theorem \ref{m-theorem} for eq.~\eqref{m-eq} with $\eta=0$, where  $\nu$ is a small
non-zero complex number such that $\Re\nu\ge0$ 	and $\Im\nu\le0$ are known, see \cite{k-gafa1999}. In particular, if $\nu$ 
is a positive  real number and $u_0$ is a smooth function of order one, then for any integer $m\ge4$
 a solution $u_\nu(t,x;u_0)$ satisfies  estimates \eqref{avb} with the averaging 
$
\nu \EE\int_t^{t+\nu^{-1}}\! \!\dots ds
$
replaced by 
$
\nu^{1/3} \int_0^{\nu^{-1/3}}\! \!\dots ds,
$
with the same upper bound and with the  lower bound $C_m \nu^{-\kappa_m m}$, where 
 $\kappa_m\to1/3$  as $m\to\infty$. Moreover, it was then shown in \cite{Bir09} that the 
 lower bounds remain true with $\kappa=1/3$,
  and that the estimates 
 $
 \sup_{t\in [0, |\nu|^{-1/3}]} \| u(t)\|_{C^m} \ge C_m |\nu|^{-m/3}, m\ge2,  
 $
 hold for smooth solutions of  equation \eqref{m-eq} with $\eta=0$ and any non-zero complex ``viscosity" $\nu$.

 The better quality of the lower bounds for solutions of the 
  the deterministic equations  is due to an extra difficulty which occurs in the stochastic case: when time grows, simultaneously with increasing of high Sobolev 
 norms of a solution,  its $L_2$-norm may decrease, which accordingly 
 would weaken the mechanism, responding to the energy transfer to 
 high modes. Significant part of the proof of Theorem~\ref{m-theorem} is devoted to demonstration that the  $L_2$-norm of a 
 solution  cannot
 go down without sending up the second Sobolev norm.

 If $\eta=0$ and $\nu=i\delta\in i\R$, then \eqref{m-eq} is a Hamiltonian PDE (the defocusing Schr\"odinger equation), and   the 
 $L_2$-norm is  its  integral of motion. If this integral is of order one, then the results of  \cite{k-gafa1999} 
 (see there Appendix~3)  imply
 that   at some point of   each time-interval of order $\delta^{-1/3}$ the $C^m$-norm of a  corresponding solution
  will become  $\ \gtrsim \delta^{-m\kappa}$ if $m\ge2$, for
  any $\kappa<1/3$. 
    Furthermore,    if $n=2$ and $\delta=1$, then 
  due to  \cite{Tao} for $m>1$ and any $M>1$ there exists a $T=T(m,M)$ and a smooth $u_0(x)$ such that 
  $\|u_0\|_m< M^{-1}$ and $\| u(T;u_0)\|_m>M$.
 \medskip
 
The paper is organized as follows.  In Section \ref{apriori-estimate}, we recall the results   from \cite{kuk1999, Kuk-N2013}
on  upper estimates for solutions of the equation \eqref{m-eq}.   
Next we show in Section~\ref{l2-norm-section} that if the noise $\eta$
  is non-degenerate, the $L^2$-norm of  a solution of eq.~\eqref{m-eq} cannot stay too small on  time
   intervals of order $\mathcal{O}(\frac{1}{\nu})$ with high probability, unless its  $H^2$-norm gets very large 
    (see Lemma~\ref{l2low}). Then in  Section~\ref{section-proof} we derive from this fact the assertion~(1) 
    of Theorem~\ref{m-theorem}, and prove  assertion~(2) and both corollaries     in Section \ref{proof-section-2}.
    
    Constants in estimates never depend on $\nu$, unless otherwise stated. For a metric space $M$ we denote by $\B(M)$ the Borel $\sigma$-algebra on $M$, and by 
    ${\mathcal P}(M)$ -- the space of probability Borel measures on $M$. By $\cD(\xi)$ we denote the law of a r.v. $\xi$. 
    \smallskip
   
 We submit  to a journal an abridged version of this  paper, where we removed the proof of the lower bounds for the classical  $C^m$-norms of solutions for eq.~\eqref{m-eq},  presented in Appendix \ref{appen-cm}, and made some other small diminutions.

\section{Solutions and  estimates}\label{apriori-estimate}

Strong 
 solutions for the equation \eqref{m-eq} are defined in the usual way:

\begin{definition} \label{def-strong-sol}Let $0<T<\infty$ and  $(\Omega, \mathcal{F},\{\mathcal{F}_t\}_{t\geqslant0}, \mathbb{P})$  be the 
filtered   probability space as in the introduction.  
 Let $u_0$ in \eqref{ic1} be a r.v., measurable in $\mathcal{F}_0$ and independent from the Wiener 
process $\xi$. Then a random process $u(t)=u(t,\cdot)\in C_0(K^{n}) $, $t\in[0,T]$, 
adapted to the filtration,    is called a strong solution of \eqref{m-eq}, \eqref{ic1}, if 
\begin{enumerate}
\item a.s.  its trajectories $u(t)$  belong to the space
\[\mathcal{H}([0,T]):=C([0,T],C_0(K^{n}))\cap L^2([0,T], H^1);\]
\item  we have 
\[u(t)=u_0+\int_0^t(\nu\Delta u-i|u|^2u)ds+\sqrt{\nu}\,\xi(t),\quad \forall t\in[0,T], \; a.s.,\]
where both sides are regarded as elements of $H^{-1}$. 
\end{enumerate}
If (1)-(2) hold for every $T<\infty$, then $u(t)$ is called a strong solution for $t\in [0,\infty)$. In this case a.s. 
 $u\in \mathcal{H}([0,\infty))$, where
$$
 \mathcal{H}([0,\infty)) = \{ u(t), t\ge0: u\mid_{[0,T]} \in  \mathcal{H}([0,T]) \;\; \forall T>0\}. 
$$
\end{definition}


Everywhere below when we talk about solutions for the problem  \eqref{m-eq}, \eqref{ic1} we assume that   the  r.v. $u_0$ is as in the definition above. 
In particular,  $u_0(x)$ may be  a non-random function. 

The global well-posedness of eq.~\eqref{m-eq} was established in \cite{kuk1999, Kuk-N2013}:

\begin{theorem}\label{existence}
For any  $u_0\in C_0(K^{n})$  the problem \eqref{m-eq}, \eqref{ic1} has a unique strong solution $u^\om(t,x;u_0)$,  $t\ge0$. 
 The family of   solutions $\{ u^\om(t;u_0)\}$  defines in the space $C_0(K^{n})$ a Fellerian Markov process. 
\end{theorem}

In  \cite{kuk1999, Kuk-N2013} the theorem above is proved when \eqref{B_m} is replaced by the weaker assumption $B_*<\infty$, where 
$\ 
B_* =\sum |b_d| 
$
(note that $B_*\le C_n B_{m_*}^{1/2}$).

  
  The transition probability for the obtained Markov process in $C_0(K^{n})$ is 
$$
P_t(u,\Gamma)=\mathbb{P}\{u(t;u)\in\Gamma\}, \quad u\in C_0(K^{n}),\; \Gamma\in \mathscr{B}(C_0(K^{n})),
$$
and the corresponding Markov semigroup in the space $\mathscr{P}(C_0(K^{n}))$ of Borel measures on $C_0(K^{n})$
 is formed by the operators $\{ \mathcal{B}_t^*, t\ge0\}$, 
\[ \mathcal{B}_t^*\mu(\Gamma)=\int_{C_0(K^{n})}P_t(u,\Gamma)\mu(du),\quad t\in\mathbb{R}.
\]
\medskip 
Then  $\mathcal{B}_t^*\mu = \cD u(t;u_0)$ if $u_0$ is a r.v., independent from $\xi$ and such that   $\cD(u_0)=\mu$.

Introducing the slow time  $\tau=\nu t$ and denoting $v(\tau,x)=u(\frac{\tau}{\nu},x)$, we rewrite  eq.~\eqref{m-eq} in the form below, more convenient for some calculations:
\begin{equation}\label{the_eq}
\frac{\partial v}{\partial \tau}-\Delta v+i\nu^{-1}|v|^2v=\tilde{\eta}(\tau,x),
\end{equation}
where \[\tilde{\eta}(\tau,x)=\frac{\partial}{\partial\tau}\tilde{\xi}(\tau,x),\quad \tilde{\xi}(\tau,x)=\sum_{d\in\mathbb{N}^n}b_d\tilde{\beta}_d(\tau)\varphi_d(x),\]
and $\tilde{\beta}_d(\tau):=\nu^{1/2}\beta_d(\tau\nu^{-1})$, $d\in \mathbb{N}^d$, is another set of independent 
standard complex Brownian motions. 

 Let $\Upsilon\in C^\infty(\mathbb{R})$ be any smooth function such 
\[\Upsilon(r)=\begin{cases}0 ,&\text{ for }r\leqslant \frac{1}{4};\\
r, &\text{ for }r\geqslant \frac{1}{2}. \end{cases}\]
Writing $v\in\C$ in the polar form $v=re^{i\Phi}$, where $r=|v|$, and recalling that $\lan\cdot,\cdot\ran$ 
stands for the real scalar product in $\C$, we apply
 It\^{o}'s  formula to $\Upsilon(|v|)$ and obtain that the process $\Upsilon(\tau):=\Upsilon(|v(\tau)|)$ satisfies 
\be\label{ups}
\begin{split} \Upsilon(\tau)&=\Upsilon_0+\int_0^\tau\Big[\Upsilon'(r)(\nabla r-r|\nabla \Phi|^2)\\
&\quad+\frac{1}{2}\sum_{d\in\mathbb{N}^n}b_d^2\Big(\Upsilon''(r)\langle e^{i\Phi},\varphi_d\rangle^2+
\Upsilon'(r)\frac{1}{r}(|\varphi_d|^2-\langle e^{i\Phi},\varphi_d\rangle^2)\Big)\Big]ds+\mathbb{W}(\tau),\end{split}
\ee
where $\Upsilon_0=\Upsilon(|v(0)|)$ and $\mathbb{W}(\tau)$ is the stochastic integral 
\[\mathbb{W}(\tau)=\sum_{d\in\mathbb{N}^n}\int_0^\tau\Upsilon'(r)b_d\varphi_d\langle e^{i\Phi}, d\tilde{\beta}_d(s)\rangle.\]

In \cite{Kuk-N2013} eq. \eqref{the_eq} is considered with $\nu=1$ and, following \cite{kuk1999}, the norm $|v(t)|_\infty$ of a solution $v$ is
estimated via $\Upsilon(t)$ (since $|v| \le \Upsilon +1/2$). But the nonlinear term $i\nu^{-1}|v|^2v$ does not contribute to 
eq.~\eqref{ups}, which is the same as  the $\Upsilon$-equation 
(2.3) in \cite{Kuk-N2013} (and as the corresponding equation in \cite[Section 3.1]{kuk1999}). 
So the estimates on $|\Upsilon(t)|_\infty$ and the resulting estimates on $|v(t)|_\infty$, obtained in \cite{Kuk-N2013}, remain true for solutions of 
 \eqref{the_eq} with any $\nu$. Thus  we get  the following upper bound  for a quadratic  exponential moments of  the 
	$L_\infty$-norms of solutions:\footnote{In \cite{kuk1999} polynomial moments of the random variables $\sup_{\tau\leqslant s\leqslant \tau+T}|v(s)|_\infty^2)$
are estimated, and in  \cite{Kuk-N2013} these results are strengthened to the exponential bounds \eqref{main_esti}.	}
	
\begin{theorem} \label{sup-norm} For any $T>0$ there are constants $c_*>0$ and $C>0$, depending only on $B_*$ and $T$, such
 that for any  r.v.   $v_0^\omega\in C_0(K^{n})$ as in Definition~\ref{def-strong-sol}, 
  any $\tau\geqslant0$ and any $c\in(0,c_*]$, a solution $v(\tau; v_0)$ of eq.~\eqref{the_eq} satisfies 
\be\label{main_esti}
\mathbb{E}\exp(c \sup_{\tau\leqslant s\leqslant \tau+T}|v(s)|_\infty^2)\leqslant C\, \EE\exp(5c\ |v_0|^2_\infty)\le\infty. 
\ee
\end{theorem}

In \cite{Kuk-N2013} the result above is proved for a deterministic initial data $v_0$. The theorem's assertion follows by averaging the result 
of  \cite{Kuk-N2013} in $v_0^\omega$. 

The estimate \eqref{main_esti} is crucial for  derivation of further properties of solutions, including the  given below  
upper bounds  for their Sobolev norms, obtained in  the work  \cite{kuk1999}. Since the scaling of the equation in \cite{kuk1999} differs 
from that in  \eqref{the_eq} and 
the result there is a bit less general than in the theorem below, a sketch of the proof is given in Appendix~\ref{app_B}. 

\begin{theorem}\label{upper-bound}
 Assume that $B_m<\infty$ for some  $ m\in\Nn$, and $v_0=v_0^\nu \in H^m\cap C_0(K^n)$ satisfies 
$$
|v_0|_\infty\le M, \quad \|v_0\|_m \le M_m \nu^{-m}, \quad 0<\nu\le1. 
$$
Then  
\be\label{sob_esti}
\mathbb{E}\|v(\tau;v_0)\|_m^2\leqslant C_{m}\nu^{-m},\quad   \forall \tau\in[0,\infty),
\ee
where $C_{M,m}$ also 
depends on $M$, $M_m$ and $B_m$, $B_{m_*}$.
\end{theorem} 

Neglecting the dependence on $\nu$, we have that if $B_m<\infty$, $m\in\Nn$,  and a   r.v.
 $v_0^\omega \in H^m\cap C_0(K^n)$ satisfies  $\EE \| v_0\|_m^2<\infty$ and $\EE \exp(c\ |v_0|^2_\infty)<\infty$ for some $c>0$,
 then eq.~\eqref{the_eq} has a solution, equal $v_0$ at $t=0$, such that 
\be\label{m_est}
\EE \| v(\tau;v_0)\|_m^2 \le e^{- t}   \EE  \| v_0\|_m^2 +C, \quad \tau\ge0, 
\ee
\be\label{m_est1}
\EE \sup_{0\le\tau\le T} \| v(\tau;v_0)\|_m^2 \le  C' , 
\ee
where $ C>0$ depend on  $c, \nu, \EE \exp(c\ |v_0|^2_\infty),  B_{m_*}$ and $B_m$, while $C'$ also depends on $\EE \| v_0\|_m^2<\infty$ and~$T$. 
 See Appendix~\eqref{app_B}.

  As it is shown  in \cite{Kuk-N2013}, the estimate \eqref{main_esti} jointly with an abstract theorem from \cite{Kuk_Shi2012}, imply that under a mild 
  nondegeneracy   assumption on the random force   the Markov process  in the space $C_0(K^{n})$,  constructed in Theorem \ref{existence}, 
  is  mixing:
  
\begin{theorem}\label{mix-measure}
For each $\nu>0$, there is an integer $N=N(B_*,\nu)> 0$ such that if $b_d\neq0$ for $|d|\leqslant N$, then the equation~\eqref{m-eq} is mixing. I.e. it has 
a unique stationary measure $\mu_\nu\in\mathscr{P}(C_0(K^{n}))$,  and for any probability measure $\lambda\in\mathscr{P}(C_0(K^{n}))$ we have $\mathcal{B}^*_t\lambda\rightharpoonup\mu_\nu$ as $t\to\infty$.
\end{theorem}

Under the assumption of Theorem  \ref{upper-bound}, for any $u_0\in H^m$ the law $\cD u(t;u_0)$ of a solution $u(t;u_0)$
is a measure in $H^m$. The mixing property in Theorem~\ref{mix-measure}  and \eqref{sob_esti} easily imply

\begin{corollary}\label{sob_mix} 
If under  the assumptions of Theorem \ref{mix-measure} 
$B_m<\infty$ for some $m>n/2$  and $u_0\in H^m$,  then $\cD(u(t;u_0)) \strela \mu_\nu$ in $\Pe(H^m)$. 
\end{corollary} 

In view of Theorems  \ref{sup-norm}, \ref{upper-bound} with $v_0=0$  and the established mixing, we have:

\begin{corollary}\label{cor-sup}
Under the assumptions of Theorem \ref{mix-measure}, if  $v^{st}(\tau)$ is  the stationary solution of the equation,  then 
\[
\mathbb{E}\exp(c_*\sup_{\tau\leqslant s\leqslant \tau+T}|v^{st}(s)|_\infty^2)\leqslant\mathcal{C},
\]
where the constant $\mathcal{C}>0$ depends only on $T$ and $B_*$. If in addition $B_m<\infty$ for some $m\in \Nn\cup\{0\}$, then 
$
\mathbb{E}\|v^{st}(\tau)\|_m^2\leqslant C_{m}\nu^{-m},
$
where $C_m$ depends on $B_*$ and $B_m$. 
\end{corollary}

Finally we note that applying  It\^{o}'s  formula to  $\| v^{st}(\tau)\|_0^2$, where $v^{st}$ is a stationary solution
 of \eqref{the_eq}, and taking the expectation we get the balance relation
\be\label{ito_stat}
\EE \| v^{st}(\tau)\|^2_1 = B_0. 
\ee
We cannot prove that $\EE \| v^{st}(\tau)\|_0^2 \ge B'>0$ for some $\nu$-independent constant $B'$, 
and cannot bound from below
the energy $\tfrac12 \EE \| v(\tau; v_0)\|^2_0$ 
of a solution $v$ by a positive $\nu$-independent quantity. Instead in next section we get a
weaker conditional  lower bound on the energies  of solutions.

\section{Conditional lower bound for the $L^2$-norm of solutions}\label{l2-norm-section} 
In this section we prove the following result:
  
\begin{lemma}\label{l2low} Let $m\geqslant2$ and $B_m<\infty$. Let $u(\tau)\in H^m$ be a solution of \eqref{the_eq}. Take any constants 
$
\chi>0, \Gamma\ge1, \tau_0\ge0,
$
and define the stopping time 
$$
\tau_\Gamma:=\inf\{\tau \ge \tau _0 : \|u(\tau )\|_2\geqslant\Gamma\}
$$
(as usual, $\tau_\Gamma =\infty$ if the set under the $\inf$-sign is empty). 
Then 
\be\label{HG}
\mathbb{E}\int_{\tau_0}^{\tau\wedge\tau_\Gamma} \mathbb{I}_{[0,\chi]} (\|u(s)\|_0)ds\leqslant 2(1+\tau)B_0^{-1} \chi \Gamma, 
\ee
for any $\tau>\tau_0$. 
Moreover,  if $u(\tau)$ is a stationary solution of the equation~\eqref{the_eq} and  $\mathbb{E}\|u(\tau )\|_2\leqslant\Gamma,$ then
\begin{equation}\label{stat2}
\mathbb{P}\Big(\{\|u(\tau)\|_0\leqslant \chi\}\Big)\leqslant2B_0^{-1}\chi\Gamma,
\quad \forall \tau \geqslant0. \end{equation}
\end{lemma}

\begin{proof} We establish  the result by adapting the proof from  \cite{Shir} (also see 
\cite[Theorem~5.2.12]{ Kuk_Shi2012}) to non-stationary solutions. 
The argument  relies on the concept of local time for semi-martingales (see e.g. \cite[Chapter VI.1]{RY} for  details of the  concept). 
By $[\cdot]_b$ we denote the quasinorm
$
[u]_b^2 = \sum_d |u_d|^2 b_d^2. 
$

  Without loss of generality we assume $\tau_0=0$, otherwise we just need to replace $u(
  \tau,x)$ by  the process $\tilde u(\tau,x):=u(\tau+\tau_0,x)$, apply the lemma with $\tau_0=0$ and with
   $u_0$ replaced by the  initial data 
   $\tilde u_0^\omega=u^\omega(\tau_0; u_0)$, and then   average  the estimate in the random $\tilde u_0^\omega$.    Let  us write the solution  $u(\tau;u_0)$ as 
$u(\tau,x)=$ $\sum_{d\in\mathbb{N}^n}u_d(\tau)\varphi_d(x)$.
For any fixed function $g\in C^2(\mathbb{R})$,  consider the process 
$$
f(\tau )=g(\|u(\tau \wedge\tau_\Gamma)\|_0^2).
$$
Since 
\[
\partial_ug(\|u\|_0^2)=2g'(\|u\|_0^2)\llangle u,\cdot\rrangle, \quad 
\partial_{uu}g(\|u\|_0^2)=4g''(\|u\|_0^2)\llangle u,\cdot\rrangle \llangle u,\cdot\rrangle+2g'(\|u\|^2_0)\llangle \cdot,\cdot\rrangle,
\]
then by It\^{o}'s formula we have 
\begin{equation}\label{ito-f}f(\tau)=f(0)+  \int_0^{\tau\wedge\tau_\Gamma}A(s)ds+\sum_{d\in\mathbb{N}^n}b_d\int_0^{\tau\wedge\tau_\Gamma}
2g'(\|u(s)\|_0^2)\lan u_d(s),
 d\beta_d(s)\ran,\end{equation}
where 
\begin{equation}\label{A-expression}
\begin{split}
 A(s)&=2g'(\|u\|_0^2)\llangle u, \Delta u-\frac{1}{\nu}i|u|^2u\rrangle +2 \sum_d b_d^2 \big( g''(\| u\|_0^2 ) |u_d|^2 +g'(\| u\|_0^2)\big) \\
&=- 2g'(\|u\|^2_0)\| u\|_1^2+ 2g''(\|u\|_0^2)
[u]_b^2
+2g'(\|u\|_0^2) B_0, \quad u=u(s). 
\end{split} \end{equation}

\noindent
{\bf Step 1: } We firstly show that for any bounded measurable set  ${G}\subset\mathbb{R}$, denoting by 
$\mathbb{I}_G$ its indicator function, we have the following equality
\begin{equation}\label{expectation-key}
\begin{split}
2\mathbb{E} \int_0^{\tau\wedge\tau_\Gamma}&\mathbb{I}_G(f(s))  \,\big(g'(\|u(s)\|_0^2)\big)^2 \! [u(s)]_b^2 ds
=\int_{-\infty}^\infty\mathbb{I}_G(a)\\&\Big[\mathbb{E}(f(\tau)-a)_+-\mathbb{E}(f(0)-a)_+-\mathbb{E}\int_0^{\tau\wedge\tau_\Gamma}\mathbb{I}_{(a+\infty)}(f(s))A(s)ds\Big]da.
\end{split}\end{equation}

Let $L(\tau,a)$, $(\tau,a)\in[0,\infty)\times\mathbb{R}$, be the local time for the semi martingale $f(\tau)$ (see e.g. \cite[Chapter VI.1]{RY}).
Since in view of \eqref{ito-f} the quadratic variation of the process $f(\tau)$ is
$$
d  \langle f,f \rangle_s = \sum_d (2g'(\| u\|_0^2) |u_d| b_d)^2 =4\big( g'(\| u\|_0^2)\big)^2 [u]_b^2,
$$
then for  any bounded measurable set  $G\subset\mathbb{R}$,   we have the following equality (known as the occupation time formula, 
see  \cite[Corollary~VI.1.6]{RY}),
\begin{equation}\label{localtime1}\int_0^{\tau\wedge\tau_\Gamma}\mathbb{I}_G(f(s))
4\big(g'(\|u(s)\|_0^2)\big)^2\, [u(s)]_b^2
ds=\int_{-\infty}^\infty \mathbb{I}_{{G}}(a)L(\tau,a)da.\end{equation}
For the local time $L(\tau,a)$, due to Tanaka's formula (see \cite[Theorem VI.1.2]{RY}) we have
\begin{equation}\label{localtime2}\begin{split}(f(\tau)-a)_+=&(f(0)-a)_+\\
&+\sum_{d\in\mathbb{N}^n} b_d\int_0^{\tau\wedge\tau_\Gamma}\mathbb{I}_{(a,+\infty)}\big(f(s)\big)2g'(\|u(s)\|_0^2)  \langle u_d (s), d\beta_d(s)\rangle \\
&+\int_0^{\tau\wedge\tau_\Gamma}\mathbb{I}_{(a,+\infty)}(f(s))A(s)ds+\frac{1}{2}L(\tau,a).
\end{split}\end{equation}
Taking expectation of both sides of \eqref{localtime1} and \eqref{localtime2} we  obtain the required  equality~\eqref{expectation-key}.
\medskip

\noindent
{\bf Step 2: }  
Let us choose ${G}=[\rho_0,\rho_1]$ with $\rho_1>\rho_0>0$,  and $g(x) =g_{\rho_0} (x)
\in C^2(\mathbb{R})$ such that $g'(x)\geqslant0$, 
$g(x)=\sqrt{x}$ for $x\geqslant \rho_0$ and  $g(x)=0$ for $x\leqslant 0$. 
Then due to the factors $\mathbb{I}_G(f)$ and $\mathbb{I}_G(a)$ in \eqref{expectation-key}, we may there 
replace $g(x)$ by $\sqrt x$, and accordingly replace 
$
g(\|u\|^2_0), g'(\|u\|^2_0)$ and  $ g''(\|u\|^2_0)
$
 by $\| u\|_0$,   $ \tfrac12\| u\|_0^{-1}$ and $ -\tfrac14\| u\|_0^{-3}$.  So the relation \eqref{expectation-key}   takes the form 
\[\begin{split}
&\mathbb{E}\int_0^{\tau\wedge\tau_\Gamma} \mathbb{I}_G(f(s))
\|u(s)\|_0^{-2} [u(s)]_b^2
=2\int_{\rho_0}^{\rho_1}\Big[\mathbb{E}(f(\tau)-a)_+-\mathbb{E}(f(0)-a)_+\Big]da\\
&-2\int_{\rho_0}^{\rho_1}\Big\{\mathbb{E}\int_0^{\tau\wedge\tau_\Gamma}\mathbb{I}_{(a,+\infty)}\Big(f(s)\Big)\Big[\frac{2}{2\|u(s)\|_0}(B_0-\|u(s)\|_1^2)
-\frac{2}{4\|u(s)\|_0^3}
[u(s)]_b^2
\Big]ds\Big\}da.
\end{split}\]
Since the l.h.s. of  the above equality is  non-negative, we have
\begin{equation}\label{main_est}
\begin{split}
&\int_{\rho_0}^{\rho_1}\Big[\mathbb{E}\int_0^{\tau\wedge\tau_\Gamma}\mathbb{I}_{(a,+\infty)}\Big(f(s)\Big)\frac{1}{\|u(s)\|_0^3}\Big(B_0\|u(s)\|_0^2-\tfrac{1}{2} [u(s)]_b^2
\Big)ds\Big]da\\
&\leqslant \int_{\rho_0}^{\rho_1} \mathbb{E}  \Big[\big( (f(\tau)-a)_+  - (f(0)-a)_+\big)
+ 
\int_0^{\tau\wedge \tau_\Gamma}\mathbb{I}_{(a,+\infty)}\big(f(s)\big)\frac{\|u(s)\|_1^2}{\|u(s)\|_0} ds\Big]da.
\end{split}
\end{equation} 
Noting that
\[
B_0\|u\|_0^2-\tfrac{1}{2} [u(s)]_b^2
 =\sum_{d\in\mathbb{N}^n} (B_0-\frac{1}{2}b_d^2)|u_d|^2\geqslant \frac{B_0}{2} \|u\|_0^2,
\]
that  by the definition of the stopping time $\tau_\Gamma$
\[
 (f(\tau)-a)_+-(f(0)-a)_+\leqslant \Gamma,
\]
and that by interpolation, \[\int_0^{\tau\wedge\tau_\Gamma}\frac{\| u(s)\|_1^2}{\|u(s)\|_0}ds\leqslant \int_0^{\tau\wedge\tau_\Gamma}\|u(s)\|_2ds\leqslant (\tau\wedge \tau_\Gamma) \Gamma,
\]
we derive from \eqref{main_est} the relation
$$
\frac{B_0}2\! \int_{\rho_0}^{\rho_1} \big( \mathbb{E}\int_0^{\tau\wedge\tau_\Gamma}\mathbb{I}_{(a,+\infty)}\big(f(s)\big) \|u(s)\|_0^{-1}
ds\big)  da \le (\rho_1-\rho_0) \Gamma \big(1 \,+ \,   \tau ) . 
$$
 When $\rho_0 \to0$, we have $g(x) \to \sqrt x$ and $f(\tau) \to \| u(\tau \wedge \tau_\Gamma)\|_0$. 
So sending $\rho_0$ to 0 and using Fatou's lemma 
 we get from the last estimate that 
$$
\int_0^{\rho_1}\mathbb{E}\int_0^{\tau\wedge\tau_\Gamma}\mathbb{I}_{(a,\infty)}\Big(\|u(s)\|_0\Big)\|u(s)\|_0^{-1}dsda\leqslant 2\rho_1(1+\tau)B_0^{-1}\Gamma.
$$
As the l.h.s. above is not smaller than 
\[\begin{split}
 \frac{1}{\chi}\int_0^{\rho_1}\mathbb{E}\int_0^{\tau\wedge\tau_\Gamma}\mathbb{I}_{(a,\chi]}(\|u(s)\|_0)dsda,
\end{split}\]
then  
\be\label{hrrr}
\frac{1}{\rho_1}\int_0^{\rho_1}
\mathbb{E}\int_0^{\tau\wedge\tau_\Gamma}\mathbb{I}_{(a,\chi]}(\|u(s)\|_0)dsda\leqslant2(1+\tau)B_0^{-1}\Gamma\,\chi.
\ee
By the monotone convergence theorem
$$
\lim_{a\to0}  \mathbb{E}\int_0^{\tau\wedge\tau_\Gamma}\mathbb{I}_{(a,\chi]}(\|u(s)\|_0)ds =
 \mathbb{E}\int_0^{\tau\wedge\tau_\Gamma}\mathbb{I}_{(0,\chi]}(\|u(s)\|_0)ds,
$$
so we get from \eqref{hrrr} that 
\be\label{hihh}
\mathbb{E}\int_0^{\tau\wedge\tau_\Gamma}\mathbb{I}_{(0,\chi]}(\|u(s)\|_0)ds\leqslant 2(1+\tau)B_0^{-1}\Gamma\chi.
\ee
\medskip

\noindent
{\bf Step 3: }  
We continue to verify that  
 \begin{equation}\label{probability0}
 \mathbb{E}\int_0^{\tau\wedge\tau_\Gamma}\mathbb{I}_{\{0\}}(\|u(s)\|_0)ds=0.
\end{equation}
To do this  let us fix any index $d\in\mathbb{N}^n$ such that $b_d\neq0$. The  process $u_d(\tau)$ is a semimartingale,
$\ 
du_d = v_d ds +b_d d\beta_d, \,
$
where  $v_d(s)$ is the $d$-th 
 Fourier coefficient of $\, \Delta u+\frac{1}{\nu}i|u|^2u\,$ for the solution  $u(\tau)=\sum_d u_d(\tau) \phi_d$   which we discuss. Consider the stopping time 
 $$
 \tau_R = \inf\{ s\le {\tau\wedge\tau_\Gamma}: |u(s)|_\infty \ge R\}. 
 $$
 Due to \eqref{main_esti} and \eqref{m_est1}, $ \mathbb{P} (\tau_R= \tau\wedge\tau_\Gamma) \to1$ as $R\to\infty$. Let us denote 
 $
 u_d^R(\tau) = u_d(\tau\wedge\tau_R). 
 $
 To prove \eqref{probability0} it suffices to verify that 
 $$
 \pi(\delta) :=   \mathbb{E}\int_0^{\tau\wedge\tau_\Gamma}\mathbb{I}_{\{ |u_d(s)| <\delta\}} ds
  \to 0\quad\text{as} \quad \delta\to0. 
 $$
 If we replace above $u_d$ by $u_d^R$, then the obtained new quantity $ \pi^R(\delta)$ differs from  $ \pi(\delta)$ at most 
 by $  \mathbb{P} (\tau_R<  \tau\wedge\tau_\Gamma)$. The process $u_d^R$ is an Ito process with a bounded drift. So by \cite[Theorem 2.2.2, p. 52]{Kry}, 
  $ \pi^R(\delta)$ goes to zero with $\delta$. Thus,  given any $\eps>0$, we firstly choose $R$ sufficiently big and then  $\delta$
  sufficiently small to achieve $\pi(\delta) <\eps$, for a suitable $\delta(\eps)>0$. So \eqref{probability0} is verified.  Jointly
   with \eqref{hihh} this proves~\eqref{HG}.
   \medskip

{\bf Step 4:} We now consider the stationary case. Let $u(\tau)$ be a stationary solution of \eqref{the_eq}. Then applying It\^o's formula to the process $f(\tau):=g(\|u(\tau)\|_0^2)$, following the same argument as   in Step 1, we obtain for any bounded measure set $G\subset\mathbb{R}$, 
\[\begin{split}
&2\mathbb{E}\int_0^{\tau}\mathbb{I}_G(f(s))  \,\big(g'(\|u(s)\|_0^2)\big)^2 \! [u(s)]_b^2\\
&=\int_{-\infty}^\infty\mathbb{I}_G(a)\Big[\mathbb{E}(f(\tau)-a)_+-\mathbb{E}(f(0)-a)_+-\mathbb{E}\int_0^{\tau}\mathbb{I}_{[a+\infty)}(f(s))A(s)ds\Big]da,
\end{split}\]
(which is exactly \eqref{expectation-key} without the stoping time $\tau_\Gamma$).
From  the stationarity of the solution, we have
\[\mathbb{E}(f(\tau)-a)_+-\mathbb{E}(f(0)-a)_+=0,\]
\[\mathbb{E}\int_0^{\tau}\mathbb{I}_G(f(s))  \,\big(g'(\|u(s)\|_0^2)\big)^2 \! [u(s)]_b^2ds=\tau\mathbb{E}\Big(\mathbb{I}_G(f(s))  \,\big(g'(\|u(s)\|_0^2)\big)^2 \! [u(s)]_b^2\Big),\;\forall s\geqslant0,\]
and
\[\mathbb{E}\int_0^{\tau}\mathbb{I}_{[a+\infty)}(f(s))A(s)ds=\tau\mathbb{E}\Big(\mathbb{I}_{[a+\infty)}(f(s))A(s)\Big),\;\forall s\geqslant0.\]
Therefore,  with \eqref{A-expression}, we have
\[\begin{split}&\mathbb{E}\int_G\mathbb{I}_{[a,\infty)}\big(g(\|u\|_0^2)\big)\Big(2g'(\|u\|_0^2)(B_0-\|u\|_1^2)+2g''(\|u\|_0^2)
[u]_b^2
\Big)da\\
&+\sum_{d\in\mathbb{N}^n}\mathbb{E}\Big(\mathbb{I}_{{G}}\big(g(\|u\|_0^2)\big)\big(2g'(\|u\|_0^2)|u_d|\big)^2\Big)=0.\end{split}\]
Then proceeding as in Step 2,  we obtain the inequality \eqref{stat2}. Thus, we finished the proof of the theorem.
\end{proof}

\section{Lower bounds for  Sobolev norms of solutions}\label{section-proof}
In this section we  work with eq.~\eqref{m-eq} in the original time scale $t$ and  provide lower bounds for the $H^m$-norms 
of its solutions with  $m>2$. This will prove the assertion (1) of Theorem \ref{m-theorem}. As always, the constants do not depend on $\nu$, unless otherwise 
stated. 

\begin{theorem}\label{t_sobolev} For any    $m\geqslant3$, if  $B_m<\infty$
  and   
  $$
  0<\kappa<\tfrac{1}{35},\quad T_0\ge0, \quad T_1>0,
  $$
  then for any   r.v. $u_0(x)\in H^m\cap C_0(K^n)$,   satisfying  
  \be\label{new_cond}
  \EE \| u_0\|_m^2<\infty, \quad \EE \exp(c\ |u_0|^2_\infty) \le C
  <\infty
  \ee
 for some $c, C >0$,   we have 
\be\label{1.0_sobolev}
\PP \Big\{ \sup_{T_0 \le t \le T_0+{T_1}{\nu^{-1}}}  \| u(t;u_0)\|_{m} \ge K \nu^{-m\kappa} 
\Big\} \to 1\quad \text{as} \quad \nu\to0, 
\ee
for every $K>0$.
\end{theorem}

\begin{proof} 
 Consider the complement to the event in \eqref{1.0_sobolev}:
\[ 
 Q^\nu=\{\sup_{T_0\leqslant t \leqslant T_0+\frac{T_1}{\nu}}\|u(t )\|_{m}<K\nu^{-m\kappa}\}.
\]
  We will prove the assertion \eqref{1.0_sobolev} by contradiction. Namely, we 
  assume that there exists a $\gamma>0$ and a sequence $\nu_j\to0$ such that 
\be\label{contra}
\PP( Q^{\nu_j})\geqslant 5\gamma \quad\text{for}\quad j=1,2,\dots,
\ee
and will derive a contradiction. 
Below we write $ Q^{\nu_j}$ as $ Q$ and always suppose that 
$$
\nu \in \{\nu_1, \nu_2, \dots\}. 
$$
The constants below may depend on $\cK, K, \gamma$, $B_{\,m\vee m_*}$, 
 but not on $\nu$.

Without lost of generality   we assume that  $T_1=1$.    For any $T_0>0$, due to \eqref{m_est} and \eqref{main_esti} the r.v.
$
\tilde u_0 := u(T_1)
$
satisfies \eqref{new_cond}  with $c$ replaced by $c/5$. So considering $\tilde u(t,x) = u(t+T_0,x)$ we may assume that $T_0=0$.

 Let us denote $J_1=[0,\frac{1}{\nu}]$. Due to Theorem \ref{sup-norm}, 
 \[\mathbb{P}(Q_1)\geqslant 1-\gamma, \quad Q_1=\{\sup_{t\in J_1}|u(t)|_\infty\leqslant C_1(\gamma)\},
 \]
 uniformly in $\nu$, for a suitable $C_1(\gamma)$. Then, by the definition of $Q$ and  Sobolev's  interpolation, 
 \be\label{hih}
 \|u^\omega(t)\|_l\leqslant C_{l,\gamma}\nu^{-l\kappa}, \quad \omega\in Q\cap Q_1, \; t\in J_1,
 \ee
 for $l\in[0,m]$ (and any $\nu\in \{\nu_1,\nu_2,\dots\}$).

 Denote $J_2=[0,\frac{1}{2\nu}]$ and consider the stopping time
 \[
 \tau_1=\inf\{t\in J_2:\|u(t)\|_2\geqslant C_{2,\gamma}\nu^{-2\kappa}\} \le \tfrac{1}{2\nu}.
 \]
 Then  $\tau_1=\frac{1}{2\nu}$ for $\omega\in Q\cap Q_1$. So due to \eqref{HG} with $\Gamma=C_{2,\gamma}\nu^{-2\kappa}$, for any $\chi>0$, we have
 \[\begin{split}\mathbb{E}\big(\nu\int_{J_2}\mathbb{I}_{[0,\chi]}(\|u(s)\|_0)ds\mathbb{I}_{Q\cap Q_1}(\omega)\big)&=\mathbb{E}\big(\nu\int_0^{\frac{1}{2\nu}\wedge\tau_1}\mathbb{I}_{[0,\chi]}(\|u(s)\|_0)ds\mathbb{I}_{Q\cap Q_1}(\omega)\big)\\
 &\leqslant \mathbb{E}\big(\nu\int_0^{\frac{1}{2\nu}\wedge\tau_1}\mathbb{I}_{[0,\chi]}(\|u(s)\|_0)ds\big)\leqslant C \nu^{-2\kappa}\chi.
 \end{split}\] 
 Consider the event
 \[\Lambda=\{\omega\in Q\cap Q_1:\|u(s)\|_0\leqslant \chi, \;\forall s\in J_2\}.
  \]
Due to the above,  we have, 
 \[\mathbb{P}(\Lambda)\leqslant   2 \mathbb{E}\big(\nu\int_{J_2}\mathbb{I}_{[0,\chi]}(\|u(s)\|_0)ds\mathbb{I}_{Q\cap Q_1}(\omega)\big)\leqslant 2{C}\nu^{-2\kappa}\chi.\]
 So $\mathbb{P}(\Lambda)\leqslant\gamma$ if we choose
 \begin{equation}\label{chi}
 \chi=c_3(\gamma)\nu^{2\kappa}, \quad c_3(\gamma)=\gamma(2{C})^{-1}.
 \end{equation}
 Let us set 
 \begin{equation}\label{q2-set}
 Q_2=(Q\cap Q_1)\setminus \Lambda, \quad \mathbb{P}(Q_2)\geqslant 3\gamma,
 \end{equation}
 and for  $\chi$ as in \eqref{chi}, consider the stopping time 
$$
 \tilde\tau_1=\inf\{t\in J_2:\|u(t)\|_0\geqslant \chi\}.
 $$
 Then $\tilde\tau_1\leqslant \frac{1}{2\nu}$ for all $\omega\in Q_2$. 
 Consider the function 
$$
v(t,x) := u (\tilde\tau_1+t,x), \quad t \in [0,\tfrac1{2\nu}]. 
$$
It solves eq.~\eqref{m-eq} with modified Wiener processes and with  initial data $v_0(x) = u^\omega(\tilde\tau_1,x)$, satisfying 
\be\label{lower_b}
\| v_0^\omega\|_0 \ge \chi=c\nu^{2\kappa} \quad \text{if}\quad \omega\in  Q_2. 
\ee
Now we 
introduce another stopping time, in terms of $v(t,x)$: 
\[\tau_2=\inf\{t\in[0,\tfrac1{2\nu}]: \|v(t)\|_m\geqslant K\nu^{-m\kappa}\}\le \tfrac1{2\nu}. 
 \]
 For $\omega\in  Q_2$,  $\tau_2=\tfrac1{2\nu}$ and in view of \eqref{hih} 
\begin{equation}\label{a-bound}
\|v^\omega(t)\|_l\leqslant C_3(\gamma)\nu^{-l\kappa}, \quad t\in[0,\tfrac1{2\nu}],\; l\in[0,m], \quad \forall\, \omega \in Q_2. 
\end{equation}

\noindent
 \textbf{Step 1}: Let us estimate from above   the increment $\mathscr{E}(t,x)=|v(t\wedge\tau_2,x)|^2-|v_0(x)|^2$. 
  Due to  It\^o's formula, we have that
\[
\begin{split}\mathscr{E}(t,x)&=2\nu\int_{0}^{t\wedge\tau_2}\Big(\langle v(s,x),\Delta v(s,x)\rangle +\sum_{d\in\mathbb{N}^n}b_d^2\varphi^2_d(x)\Big)ds + \sqrt{\nu} \,  M(t,x), \\
& M(t,x) = \int_{0}^{t\wedge \tau_2}\sum_{d\in\mathbb{N}^n}b_d\varphi_d(x)\langle v(s,x),d\beta_d(s)\rangle. \, 
\end{split}
\]

We treat $ M$ as a martingale $M(t)$ 
in the space $H^1$. Since in view of  \eqref{a2} for $0\le s<\tau_2$ we have 
$$
\| v(s) \phi_d\|_1 \le C\big( |v(s)|_\infty \| \phi_d\|_1 +  |v(s)\|_1 |\phi_d|_\infty \big) \le
 C(\zeta d+\zeta^{(m-1)/m} \nu^{-\kappa}), 
$$
  where $\zeta =  \sup_{0\leqslant s\leqslant \frac{1}{\nu}}|u(s)|_\infty $
(the assertion is empty if $\tau_2=0$), then for any $0<T_*\le \tfrac1{2\nu}$ 
 \be\label{Ito_int}
 \mathbb{E}\| M(T_*)\|_1^2 \le
 \int_0^{T_*}  
  \EE\, \sum_d b_d^2 \| \phi_d v(s)\|_1^2 ds \le C T_* \nu^{-2\kappa}, 
   \ee
 where we  used that $B_1<\infty$.  So
 by  Doob's inequality 
\be\label{Doob1}
\mathbb{P}\Big(\sup_{0\leqslant s\leqslant T_*} \| M(s)\|_1^2\geqslant r^2\Big)\leqslant C T_*r^{-2}  \nu^{-2\kappa},
 \quad \forall r>0.
\ee

Let us choose  
$$
T_*=\nu^{-b}, \quad b\in (0, 1),
$$ 
where $b$ will be  specified  later. 
Then
$
1\le T_*\le \tfrac1{2\nu}
$
if $\nu$ is sufficiently small, so due to \eqref{Doob1} 
\[
\mathbb{P}( Q_3)\geqslant1-\gamma,  \quad Q_3=\{\sup_{0\leqslant\tau\leqslant T_*}\|  M(\tau)\|_1\leqslant C_4(\gamma)\nu^{-\kappa}\sqrt{T_*}\},
\]
for a suitable $C_4(\gamma)$ (and for $\nu\ll1$);  thus   $ \mathbb{P} (Q_2\cap Q_3) \ge 2\gamma$. 
Since $ \| \lan v , \Delta v\ran \|_1 \le C|v|_\infty \| v\|_3$ by \eqref{a3} and 
$
 \| \sum_d  b_d \phi_d\|_1 \le C, 
$
then in view of \eqref{a-bound}  and the definition of $Q_3$, 
\begin{equation}\label{l2-upper}
\|\mathscr{E}^\omega(\tau)\|_1\leqslant C(\gamma)(\nu^{1-3\kappa}T_*+\nu^{\frac{1}{2}-\kappa} T_*^{1/2}), \qquad \forall\, 
 \tau\in[0,T_*] , \;\; \forall\,\omega\in Q_2\cap Q_3.
\end{equation}
\smallskip

\noindent
 \textbf{Step 2}: For any $x\in K^n$, denoting 
 $
 R(t) = |v(t,x)|^2, \, a(t) = \Delta v(t,x)
 $
 and $\xi(t) = \xi(t,x)$, 
 we write the equation for $v(t) := v(t,x)$ as an It\^o process:
 \be\label{v_ito}
 dv(t) = (-i R v +\nu a) dt + \sqrt\nu\, d\xi(t).
 \ee
  Setting $w(t) = e^{i\int_0^t R(s)ds} v(t)$, we observe that $w$ also is an It\^o process, $ w(0) =v_0$  and 
 $
 dv = e^{-i\int_0^t R(s)ds} dw -i Rv\, dt.
 $
 From here and \eqref{v_ito}, 
 $$
 w(t) = v_0 +\nu \int_0^t e^{i\int_0^s R(s')ds'} a(s) ds + \sqrt\nu\, \int_0^t e^{i\int_0^s R(s')ds'}  d \xi(s). 
 $$
 So $v(t\wedge \tau_2)= v(t\wedge\tau_2,x)$ can be written as  
 \be\label{deco}
v(t\wedge \tau_2,x)=  I_1(t\wedge \tau_2,x) + I_2(t\wedge \tau_2,x)+ I_3(t\wedge \tau_2,x),
\ee
where 
\[\begin{split}
I_1(t,x)= e^{-i\int_0^{t}|v(s,x)|^2ds}v_0,\quad 
I_2(t,x)=\nu\int_0^{t}e^{-i\int_{s}^{t}|v(s',x)|^2ds'}\Delta v(s,x)ds, \\
I_3(t,x)=
\sqrt{\nu} e^{-i\int_0^{t}|v(s',x)|^2ds'}
\int_0^{t}e^{i\int_0^{s}|v(s',x)|^2ds'}d\xi(s,x) . 
\end{split}
\]

Our next goal is to obtain a lower bound for $ \| v(T_*)\|_1$ when $\om \in Q_2\cap Q_3$, 
using the above decomposition \eqref{deco}. 
\medskip

\noindent
 \textbf{Step 3}: We first deal with the stochastic term    $I_3(t)$.  For $0 \leqslant s\leqslant s_1\leqslant T_*\wedge\tau_2$ we set
 \begin{equation}\label{W-def}
 \begin{split}
 W(s,s_1, x):= \exp( i  \int_{s}^{s_1}|v(s',x)|^2ds'), \quad 
 F(s,s_1,x) :=
 \int_{s}^{s_1}|v(s',x)|^2ds';
 \end{split}
 \end{equation}
 then  $ W(s,s_1, x) = \exp \big(i F(s,s_1, x)\big)$. The functions $F$ and $W$ are periodic in $x$, but not odd. Speaking about them
 we understand $\|\cdot\|_m$ as the non-homogeneous Sobolev norm, so
 $
 \| F\|_m^2  =  \| F\|_0^2 +  \|( -\Delta)^{m/2}  F\|_0^2,
 $
 etc.  We  write $I_3$ as 
 \be\label{I_3}
 I_3(t) =\sqrt\nu\  \overline W ( 0,t \wedge  \tau_2 ,x) \int_0^{t\wedge \tau_2} W(0,s,x) d\xi(s,x).
 \ee
 \noindent
 In view of \eqref{a1},  
  \be\label{esti}
  \| \exp(iF(s,s_1 \cdot))\| _{k} \le C _{k}(1+ |F(s,s_1, \cdot)|_\infty)^{k-1} \| F(s,s_1,\cdot)\| _{k}, \quad k \in \Nn.
  \ee
\noindent

For any $s\in J =[ 0, T_*\wedge \tau_2)$, by \eqref{a2} and 
 the definition of $ \tau_2$, we have that $v:=v(s)$ satisfies 
\be\label{hoh}
 \| |v|^2\|_{1} \le C |v|_\infty \| v\|_{1}  \le C |v|_\infty \| v\|_0^{1- 1/m} \| v\|_m^{1/m} 
 \le  C' |v|_\infty^{2-1/m} \nu^{-\kappa}
\ee
 (this assertion is empty if $ \tau_2=0$ since then $J=\emptyset$). 
So for $s,s_1 \in J$, 
$$
|F(s,s_1,\cdot)|_\infty\leqslant   |s_1-s| \sup_{s'\in J}|v(s')|^2_\infty,\quad \|F (s,s_1,\cdot)\|_{k}
 \le C  \nu^{ -\kappa k} |s_1-s|\big(\sup_{s' \in J} | v(s')|_\infty\big)^{2-k/m}\;\;
$$
for $k\le m$. 
Then, due to \eqref{esti}, 
   \be\label{w_est}
   \|W(0,s\wedge\tau_2,\cdot)\|_{1} \leqslant  C'  T_* \nu^{-\kappa} (1+ \sup_{s \in J} |v(s)|_\infty^{2} ).
   \ee
     Consider the stochastic integral in \eqref{I_3}, 
   $$
    N(t,x)=\int_0^{t}W(0,s,x) d\xi(s,x).
   $$
  The process
  $
  t \mapsto W(0,t,x)
  $
  is adapted to the filtration $\{ \cF_t\}$, and 
   \[dW(0,t,x)=i|v(t,x)|^2W(0,t,x)dt.\]
   So integrating by parts (see, e.g., \cite[Proposition IV.3.1]{RY}) we re-write $N$ as 
   \[N(t,x)=W(0,t,x)\xi(t,x)- i \int_0^t\xi(s,x)  |v(s,x)|^2W(0,s,x)ds,
   \]
   and we see from  \eqref{I_3} that 
   \begin{equation}\label{I_3-2}I_3(t)=\sqrt{\nu}\xi(t\wedge\tau_2,x)+ i\sqrt{\nu}\int_0^{t\wedge\tau_2}\xi(s,x)  |v(s,x)|^2W(s,t\wedge\tau_2,x)ds.
   \end{equation}
  Due to \eqref{B_m} and since $B_m<\infty$, the   Wiener process $\xi(t,x)$  satisfies 
  \[\mathbb{E}\|\xi(T_*,x)\|_1^2\leqslant CB_1T_*,\]
  and
  \[\mathbb{E}\sup_{0\leqslant t\leqslant T_*}|\xi(t,\cdot)|_\infty\leqslant \sum_{d\in\mathbb{N}^n}b_d(\mathbb{E}\sup_{0\leqslant t\leqslant T_*}|\beta_d(t)\varphi_d|_\infty)\leqslant CB_*\sqrt{T_*},\]
  (we recall that $B_*=\sum_{d\in\mathbb{N}^n}|b_d|<\infty$).  
  Therefore,
   \[ \mathbb{P}(Q_4)\geqslant 1-\gamma, \quad Q_4=\{ \sup_{0\leqslant t\leqslant T_*}
   ( \|\xi(t)\|_1\vee |\xi(t)|_\infty)
   \leqslant CT_*^{1/2}\},\]
 with a suitable $C=C(\gamma)$. 
Let 
$$
\tilde Q=\bigcap_{i=1}^4 Q_i,
$$ 
then $\PP(\tilde Q)\geqslant\gamma$.  As $\tau_2 = T_*$ for $\omega \in \tilde Q$, then due  to  \eqref{hoh},   \eqref{w_est},
\eqref{I_3-2} and \eqref{a2},
 for $\omega\in \tilde Q$ we have 
\begin{equation}\label{i3-upper}
\begin{split}
\sup_{0\leqslant t\leqslant T_*} \| I^\omega_3(t)\|_1 &\le   \sqrt\nu\,  \sup_{0\leqslant t\leqslant T_*} 
\Big(\|\xi^\omega(t)\|_1+\int_0^t\|\xi^\omega(s)|v^\omega(s)|^2W^\omega(s,t)\|_1ds\Big)\\
 &\leqslant   {C} T_*^{5/2}\nu^{\frac{1}{2}-\kappa}.
 \end{split}
 \end{equation}
 
 \noindent
\textbf{Setp 4}: We then consider the term 
$
I_{2}= \nu\int_0^{t\wedge \tau_2} \bar W(s, t\wedge \tau_2,x) 
\Delta v(s,x)ds
$. 
To bound   its $H^1$-norm we need to estimate $\|W\Delta v\|_1$. 
Since 
\[\| \partial_x^a W\partial_x^bv\|_0\leqslant C \|W\|_3^{1/3}\|v\|_3^{2/3}|v|_\infty^{1/3}\quad \text{if} \;\;
 |a|=1,|b|=2,\]
(see \cite[Proposition  3.6]{Tay}),
we have 
\[\|W\Delta v\|_1\leqslant C(\|v\|_3+ \|W\|_3^{1/3}\|v\|_3^{2/3}|v|_\infty^{1/3}).\]
Then in view of \eqref{esti} and \eqref{a-bound},   for $\omega\in\tilde Q$
$$
\|W\Delta v\|_1 \le C\big( \nu^{-3\kappa} +(T_*^3 \nu^{-3\kappa})^{1/3} \nu^{-2\kappa} \big)
 \le C \nu^{-3\kappa} T_*,
$$
and accordingly 
\begin{equation}\label{i1-upper}
\sup_{0\leqslant t\leqslant T_*}\| I^\omega_2(t)\|_1\leqslant \nu \sup_{0\leqslant t\leqslant T_*}\!\int_0^{t}\|W^\omega(s,T_*)\Delta v^\omega(s)\|_1ds 
\leqslant C \nu^{1-3\kappa}T_*^{2},\;\; \;
\forall\,\omega\in\tilde Q.
\end{equation}
\smallskip

\noindent
{\bf Step 5: } Now we estimate from below  the $H^1$-norm of the term $I^\omega_{1}(T_*,x)$, $\omega\in \tilde  Q$.
Writing it as $
 I^\omega_1(T_*,x) = e^{-i T_*|v_0(x)|^2} e^{ -i \int_0^{T_*}  \mathscr{E} (s,x) ds} v_0(x)
$
wee see that
 \[
\| I^\omega_1(T_*) \|_1\geqslant \|\nabla(\exp(-iT_*|v_0|^2)v_0\|_0-\|\nabla(\exp(-i\int_0^{T_*} \mathscr{E} (s)ds))v_0\|_0-\|v_0\|_1.
\]
This first term on the r.h.s is 
\[T_*\|v_0\nabla(|v_0|^2)\|_0=T_*\tfrac{2}{3}\|\nabla|v_0|^3\|_0\geqslant CT_* \||v_0|^3\|_0\geqslant CT_*\|v_0\|_0^3\ge
C T_* \nu^{6\kappa},\quad C>0,
\]
where we have used the fact that $u|_{\partial K^{n}}=0$,   Poincar\'e's inequality and \eqref{lower_b}.

For $\om\in \tilde Q$ and $0\le s\le T_*$, in view of \eqref{l2-upper}, the second term is bounded by
\[
\begin{split}
\|(\int_0^{T_*}\nabla \mathscr{E}(s)ds)v_0\|_0  \leqslant CT_*|v_0|_\infty \sup_{0\leqslant s\leqslant T_*}\|\mathscr{E}(s)\|_1
 \le
{C} T_*( \nu^{1-3\kappa} T_* + \nu^{\frac{1}{2}-\kappa} T_*^{1/2} ). 
\end{split} 
\]
Therefore, using \eqref{l2-upper}, we get for the term $I^\omega_1(T_*)$ the following lower bound:
 \[
 \|I_{1}^\omega(T_*)\|_1\geqslant C  \Big(\nu^{6\kappa}T_*-  T_*\big(\nu^{1-3\kappa}T_* 
+\nu^{\frac{1}{2}-\kappa} T_*^{1/2}\big) 
-\nu^{-\kappa} \Big).
\]
Recalling $T_*=\nu^{-b}$ we see that if  we assume that 
\begin{equation}\label{restrict1}
\begin{cases}
6\kappa-b<-\kappa,\\
6\kappa -b < 1-3\kappa -2b, 
\\
 6\kappa-b < 
1/2 -\kappa-\frac{3}{2}b , 
\end{cases}
\end{equation}
then for $\omega\in\tilde Q$, 
\begin{equation}\label{i2-below1}\|I_{1}^\omega(T_*)\|_1\geqslant C \nu^{6\kappa}T_*, \quad  C >0, 
\end{equation}
provided that $\nu$ is sufficiently small. 
\smallskip

\noindent
{\bf Step 6: }Finally, remembering that $\tau_2=T_*$ for $\omega \in \tilde Q$ and 
combining the relations  \eqref{i3-upper}, \eqref{i1-upper} and \eqref{i2-below1} to estimate the terms of \eqref{deco}, we see that 
for $\omega\in\tilde Q$ we have
\be\label{new_esti} 
\|v^\omega(T_*)\|_1\geqslant\|I_{1}^\omega(T_*)\|_1-\|I^\omega_2(T_*)\|_1-\|I^\omega_3(\tau_*)\|_1\geqslant\tfrac{1}{2}C_1
\nu^{6\kappa-b}, \quad  C_1>0, 
\ee
if we assume in addition to \eqref{restrict1} that 
\begin{equation}\label{restrict2}
6\kappa-b<\frac{1}{2}-\kappa-\frac{5}{2}b,
\end{equation}
and $\nu$ is small. 
Note that this relation  implies
the last two  in \eqref{restrict1}.

Combining  \eqref{a-bound} and \eqref{new_esti} 
we get that 
\begin{equation}\label{restrict3}\nu^{-b+7\kappa}\leqslant C^{-1}_2,
\end{equation}
for all sufficiently small $\nu$. 
Thus we  have obtained a contradiction with the existence of the sets $Q^{\nu_j}$ 	as at the beginning of the proof 
 if (for a chosen $\kappa$) we can find a $b\in(0,1)$ which  
  meets \eqref{restrict1},  \eqref{restrict2} and
$$
-b+7\kappa<0.
$$
Noting that this is nothing but the first relation in  \eqref{restrict1}, we see that  we have obtained a contradiction if 
$$
\kappa< \tfrac17 b, \quad  \kappa< \tfrac1{14} -\tfrac3{14} b,
$$
for some $b\in(0,1)$. 
We see immediately that such a $b$ exists if  and only if $\kappa < \tfrac1{35} $.

\end{proof}

\noindent
{\bf Amplification.} If we replace the condition $m\geqslant3$ with the weaker assumption 
\[
\R\ni  m>2,
\]
then the  statement \eqref{1.0_sobolev} remains true for $0<\kappa<\kappa(n,m)$ with a suitable (less explicit) constant 
$\kappa(n,m)>0$. In this case  we obtain a contradiction with the assumption \eqref{contra} 
 by deriving a lower bound for $\|v(T_*)\|_\alpha$,  where 
$\alpha=\min\{1,m-2\}$, using the decomposition \eqref{deco}.  
The proof remains almost identical except that  now, firstly, we 
 bound $\|I_{2}\|_\alpha$ ($\alpha<1$) from above 
using the following estimate from \cite[Theorem 5, p.~206]{RS} (also see there p.~14):
\[\|W\Delta u\|_\alpha\leqslant C\|u\|_{2+\alpha}(|W|_\infty+|W|_\infty^{1-\frac{2\alpha}{n}}\|W\|_2^{\frac{2\alpha}{n}});\]
and, secondly,    estimate $\|I_{1}^\omega(T_*)\|_\alpha$ ($\alpha<1$) from below as 
\[
\|I_{1}^\omega(T_*)\|_\alpha\geqslant {\|I_{1}^\omega(T_*)\|_1^{2-\alpha}} \,{\|I^\omega_2(T_*)\|_2^{-1+\alpha}},
\]
which directly follows from Sobolev's  interpolation. The lower bound for $\|I^\omega_2(T_*)\|_1$ in \eqref{i2-below1} stays valid,  so to bound $\|I_2^\omega(T_*)\|_\alpha$ $(\alpha<1)$ from below, we just need to estimate the upper bound of $\|I_2^\omega(T_*)\|_2$.  Since  $I^\omega_2(T_*,x)=\overline W(0,T_*,x)v_0(x)$ (see \eqref{W-def}), 
then in view of \eqref{a2}  we have 
\[\|I_2^\omega(T_*)\|_2\leqslant C (|\overline W(T_*)|_{\infty}\|v_0\|_2+\|\overline W(T_*)\|_2|v_0|_\infty).\]
So by \eqref{w_est}, 
$
\|I_2^\omega(T_*)\|_2\leqslant C_{8,\gamma}\nu^{-2\kappa}T_*^2.
$
Therefore,
\begin{equation}\label{i2-below}
\|I_2^\omega(T_*)\|_\alpha\geqslant C_{9,\gamma}\frac{(\nu^{6\kappa}T_*)^{2-\alpha}}{(\nu^{-2\kappa}T_*^2)^{1-\alpha)}}\geqslant C_{9,\gamma}\nu^{(14-8\alpha)\kappa-b\alpha}.\end{equation}
We then can complete the proof by an argument,  similar to that at  Step 6.

\begin{remark}\label{stationary-lower} If  eq. \eqref{m-eq} is mixing and $u(t)$ is a stationary solution, then if $\mathbb{E}\|u(t)\|_2\leqslant \nu^{-2\kappa}$, we have for $m\geqslant3$,
\begin{equation}\label{lower-stat1}\mathbb{P}\Big\{\sup_{T_0\leqslant t\leqslant T_0+T_1\nu^{-7.001\kappa}}\|u(t)\|_m\geqslant K\nu^{-m\kappa}\Big\}\to1\quad\text{as}\quad \nu\to0.\end{equation}
Indeed, due to \eqref{stat2}, in the stationary case, we can choose $Q_2$ as in \eqref{q2-set} such that for the stopping time $\tilde \tau_1$ defined in \eqref{stop-stat}, $\tilde\tau_1(\omega)=0$, if $\omega\in Q_2$. Then the same argument gives the above assertion. 
\end{remark}

\medskip

 \section{Lower bounds for time-averaged Sobolev norms}\label{proof-section-2}
In this section we prove the assertion (2)  of Theorem \ref{m-theorem}.  We provide a space $H^r$, $r\ge0$, with the scalar product
 $$
 \llangle u,v\rrangle_r:=\llangle (-\Delta)^{\frac{r}{2}}u,(-\Delta)^{\frac{r}{2}}v\rrangle,
 $$
 corresponding to the norm $\|u\|_r$.   Let 
$
u(t) =\sum u_d(t) \phi_d
$
 be a solution of eq.~\eqref{m-eq}.  Applying It\^{o}'s formula to the functional $\|u\|_m^2$,
 we have for any $0\leqslant t<t'<\infty$ the relation
\begin{equation}\label{balance1}\begin{split}\|u(t')\|_m^2=&\|u(t)\|_m^2+2\int_{t}^{t'}\llangle u(s), \nu \Delta u(s) - 
i|u(s)|^2u(s)\rrangle_mds\\
&+2\nu B_m(t'-t)+2\sqrt{\nu} M(t,t'),
\end{split}\end{equation}
where $M$ is the stochastic integral 
\[M(t,t'):= \int_t^{t'}\sum_{d\in\mathbb{N}^n} b_d|d|^{2m} \langle u_d(s), d\beta_d(s)\rangle .
\]
Let us fix a $\gamma\in(0,\frac{1}{8})$.  Due to Theorem \ref{sup-norm} and  \ref{t_sobolev}, for small enough $\nu$ there exists an event 
$\Omega_1\subset\Omega$,
$
\PP(\Omega_1)\geqslant1-\gamma/2,
$
such that  for all $\om \in \Omega_1$ we have: 

a) $\sup_{0\leqslant t\leqslant\frac{1}{\nu}}|u^\omega(t)|_\infty\leqslant C(\gamma)$, for a suitable $C(\gamma)>0$;

b) there exist  $t_\omega\in[0,\frac{1}{3\nu}]$ and  $t_\omega'\in[\frac{2}{3\nu},\frac{1}{\nu}]$ satisfying
\be\label{relation}
\|u^\omega(t_\omega)\|_m,\;\|u^\omega(t_\omega')\|_m\geqslant \nu^{-m\kappa}.
\ee

Since for  the martingale $M(0,t)$  we have that 
$$
\mathbb{E}|M(0,\tfrac{1}{\nu})|^2\leqslant B_m \mathbb{E}\int_0^\frac{1}{\nu}  \|u(s)\|_m^2ds =: X_m, 
$$
then by Doob's inequality
\[
\PP(\Omega_2)\geqslant1-\frac{\gamma}{2}, \qquad 
\Omega_2=\Big\{ \sup_{0\leqslant t\leqslant\frac{1}{\nu}}|M(0,t)|\leqslant c(\gamma)
X_m^{1/2}
\Big\}.
 \]
Now let us set 
$\ \hat\Omega=\Omega_1\cap\Omega_2.$
Then 
$\PP(\hat\Omega)\geqslant1-\gamma$ for small  enough $\nu$, and for 
 any  $\omega\in \hat\Omega$ there are two alternatives:

i) there exists a $t_\omega^0\in[0,\frac{1}{3\nu}]$ such that $\|u^\omega(t^0_\omega)\|_m=\frac{1}{3}\nu^{-\kappa m}$.  Then 
from \eqref{balance1} and \eqref{relation} in view of \eqref{a4} we get 
\[\begin{split}
  \frac{8}{9}\nu^{-2m\kappa}+2\nu \int_{t_\omega^0}^{t'_\omega}\|u^\omega(s)\|_{m+1}^2ds 
\leqslant C(m,\gamma)  \int_0^{\frac{1}{\nu}} \|u^\omega(s)\|_m^2ds+2B_m+2\sqrt{\nu}c(\gamma)
X_m^{1/2}.
\end{split}\]

ii) There  exists no $t\in[0,\frac{1}{3\nu}]$ with $\|u^\omega(t)\|_m=\frac{1}{3}\nu^{-\kappa m}$. In this case, 
since $\|u^\omega(t)\|_m$ is continuous with respect to $t$, then due to \eqref{relation} 
$
\|u^\omega(t)\|_m > \frac{1}{3}\nu^{-m\kappa}
$
for all 
$ t\in[0,\frac{1}{3\nu}]$, 
which  leads to the relation 
\[
\tfrac{1}{27}\nu^{-2m\kappa-1}\leqslant \int_0^{\frac{1}{\nu}}\|u^\omega(s)\|_m^2ds.
\]

In both cases  for $\omega\in\hat\Omega$ we have:
\[\frac{1}{27}\nu^{-2m\kappa}\leqslant C'(m,\gamma)  \int_0^{\frac{1}{\nu}}
\|u(s)\|_m^2ds+2B_m
+{\nu} c(\gamma)^2 + X_m. 
\]
This implies that 
\[\mathbb{E}\nu\int_0^{\frac{1}{\nu}}\|u(\tau)\|_m^2d\tau\geqslant C\nu^{-2m\kappa+1}
\]
(for small enough $\nu$), and 
 gives the lower bound in \eqref{avb}. The upper bound follows directly from Theorem~\ref{upper-bound}.
 \medskip
 
  \noindent
  {\it Proof of Corollaries \ref{c_3} and \ref{c_4}}:
  Since $B_k<\infty$ for each $k$ and all coefficients $b_d$ are non-zero, then eq.~\eqref{m-eq} is 
   mixing in the spaces $H^M$,  see 
Theorem~\ref{mix-measure}.  As the stationary solution $v^{st} $ satisfies Corollary~\ref{cor-sup} with any $m$, then for 
each $\mu\in \Nn$ and $M>0$, interpolating the norm $\|u\|_\mu$ via $\| u\|_0$ and $\| u\|_m$ with $m$ sufficiently 
large 
 we get that the stationary measure $\mu_\nu$  satisfies 
\be\label{mu_nu}
\int \| u\|_\mu^M \mu_\nu(du) < \infty \quad \forall\, \mu\in \Nn, \;\forall M>0. 
\ee
Similar, in view of \eqref{m_est} and Theorem  \ref{sup-norm}, 
\be\label{mu_int}
\EE \| u(t;u_0)\|_\mu^M \le C_\nu(u_0) \quad \forall t\ge0,
\ee
for each $u_0\in C^\infty$ and every $\mu$ and $M$ as in \eqref{mu_nu}. 
 Now let us consider the integral in \eqref{avb} and write it as 
$$
J_t := \nu \int_t^{t+\nu^{-1}} \EE \| u(s)\|_m^2ds.
$$
Replacing the integrand in $J_t$ with $\EE( \| u_\nu(s)\|_m \wedge N)^2$, $ N\ge1$, using the convergence  
\be\label{weak}
\EE\big(  \|u(s;v_0) \|_m  \wedge N\big)^2 \to \int \big(  \|u \|_m  \wedge N\big)^2 \mu_\nu(du)
\quad \text{as} \quad s\to\infty \quad \forall\,N, 
\ee
which follows from Corollary \ref{sob_mix},  and the  estimates  \eqref{mu_nu},  \eqref{mu_int} we get that 
\begin{equation}\label{weak2}
J_t \, \rightarrow\, \int \| u\|_m^2 \,\mu_\nu(du)\quad \text{as} \quad  t\to\infty. 
\end{equation}
This convergence and \eqref{avb} imply the assertion of  Corollary~\ref{c_3}. 
\smallskip

Now convergence  \eqref{weak} jointly with estimates   \eqref{mu_nu},  \eqref{mu_int}  and \eqref{conver} imply  Corollary~\ref{c_4}.

\appendix
\section{Some estimates}\label{section-appendix}

    For any integer $l\in \Nn$ and $F\in H^l$ 
    we have that 
  \be\label{a1}
  \| \exp(iF(x))\|_{l} \le C_{l} (1+ |F|_\infty)^{l-1} \| F\|_{l}. 
  \ee
\noindent
Indeed, to verify \eqref{a1} it suffices to check that for any non-zero  multi-indices $\beta_1,\dots,\beta_{l'}$, where $1\le l'\leqslant l$ and
 $
 |\beta_1| +\cdots+ |\beta_{l'}| =l,
 $
we have 
   \be\label{a3}
   \| \p_x^{\beta_1} F\cdots \p_x^{\beta_{l'}}  F
 \|_0 \le C  |F|_\infty^{l'-1} \| F\|_l. 
 \ee
 But this is the assertion of Lemma 3.10 in \cite{Tay}. Similarly, 
 \be\label{a2}
\|F G \|_r
\leqslant C_r ( |F|_\infty\|G\|_r + |G|_\infty\|F\|_r), \quad F,G \in H^r,\;\; r\in \Nn, 
\ee
see \cite[Proposition 3.7]{Tay} (this relation is known as Moser's estimate).  Finally, since for $| \beta| \le m$ we have 
$
| \p_x^\beta v|_{2m/ \beta|} \le C |v|_\infty^{1- |\beta| /m} \| v\|_m^{|\beta| /m}
$
(see relation (3.17) in \cite{Tay}), then 
\be\label{a4}
|\llangle |v|^2 v, v\rrangle_m| \le C_m \|v\|_m^2 |v|_\infty^2, \qquad 
|\llangle |v|^2 v, v\rrangle_m| \le C'_m \|v\|_{m+1}^{\frac{2m}{m+1}}  |v|_\infty^{\frac{2m+4}{m+1}} .
\ee

\section{Proof of Theorem \ref{upper-bound}}\label{app_B}
Applying  Ito's formula to a solution $v(\tau)$ of eq. \eqref{the_eq} we get a slow time  version of the relation \eqref{balance1}:
\begin{equation}\label{b10}
\begin{split}
\|v(\tau )\|_m^2=\|v_0\|_m^2+2 \int_{0}^{\tau } \big( -\| v \|^2_{m+1} - \nu^{-1} \llangle
i|v |^2v , v \rrangle_m \big)ds 
+2 B_m \tau+2 M(\tau ),
\end{split}\end{equation}
where
$M(\tau )= \int_0^{\tau }\sum_{d} b_d|d|^{2m} \langle v_d(s), d\beta_d(s)\rangle .$ Since in view of \eqref{a4} 
$$
 \EE \big| \llangle  |v|^2v, v\rrangle_m \big| \le 
  C_m \big( \EE \| v\|_{m+1}^{2}\big)^{\frac{m}{m+1}}  \EE\big( |v|_\infty^{2m+4}\big)^{\frac{1}{m+1}},
$$
then denoting
$
\EE \| v(\tau)\|_r^2 =: g_r(\tau), \ r \in \Nn \cup\{0\}, 
$
 taking expectation of \eqref{b10},  differentiating the result and using \eqref{main_esti}, we get that 
\be\label{b1}
\frac{d}{d\tau} g_m \le -2 g_{m+1} +C_m \nu^{-1} g_{m+1}^{\frac{m}{m+1}} +2B_m \le -2 g_{m+1}\big( 1- C'_m  \nu^{-1} g_{m}^{-\frac{1}{m}} +2B_m\big),
\ee
since 
$
g_m \le g_0^{1/(m+1)}  g_{m+1}^{m/(m+1)}  \le C_m g_{m+1}^{m/(m+1)} . 
$
We see that if $g_m \ge (2\nu^{-1} C'_m)^m$, then the r.h.s. of \eqref{b1} is 
\be\label{b2}
\le -g_{m+1} +2B_m \le -C_m^{-1} g_m^{(m+1)/m} +2B_m  \le - \bar C_m \nu^{-m-1} + 2B_m,
\ee
which is negative if $\nu\ll1$. So if
\be\label{b3}
g_m(\tau) <  (2\nu^{-1} C'_m)^m
\ee
at $\tau=0$, then \eqref{b3} holds for all $\tau\ge0$ and \eqref{sob_esti} follows. If  $g_m(0)$ violates \eqref{b3}, then 
 in view of \eqref{b1} and \eqref{b2}, 
for $\tau\ge0$,  while \eqref{b3} is false,  we have that 
$$
\frac{d}{d\tau} g_m \le -C_m  g_{m}^{(m+1)/m} +2B_m, 
$$
which again implies \eqref{sob_esti} (see details of this argument  in the proof of Theorem 2.2.1 in \cite{BK}). 
Note that in view of \eqref{b1},
$$
\frac{d}{d\tau} g_m \le - g_{m} +C_m (\nu, |v_0|_\infty, B_{m_*},B_m).
$$
This relation immediately implies \eqref{m_est}.

\medskip
Now let us return to eq. \eqref{b10}. Using Doob's inequality and   \eqref{sob_esti} we find that 
$$
\EE(\sup_{0\le \tau\le T} |M(\tau)|^2 ) \le C<\infty. 
$$
Next, applying \eqref{a4} and Young's inequality we get that 
$$
\int_{0}^{\tau } \big( -\| v \|^2_{m+1} - \nu^{-1} \llangle i|v |^2v , v \rrangle_m \big)ds \le
C_m \int_{0}^{\tau } | v(s)|_\infty^{2m+3} ds, \quad \forall \ 0\le\tau\le T.
$$
 Finally, 
using in \eqref{b10} the last two displayed formulas jointly with \eqref{main_esti} we obtain \eqref{m_est1}.

\section{Lower bound for  $C^m$-norms of solutions}\label{appen-cm}
In this appendix we work with eq. \eqref{the_eq}, in the time scale $\tau$. Our goal  is to prove the following result:

\begin{theorem} If  $m\geq 2$ is an integer %
 and $\kam<\tfrac1{16}$, then for any $\tau_0\ge0$ and $\tau'>0$, every  solution $u(\tau, x)$ 
of \eqref{the_eq} with a smooth initial data $u_0(x)$ 
 satisfies 
\be\label{1.0}
\PP \Big\{ \sup_{\tau_0 \le\tau \le\tau_0+\tau'}  | u(\tau)|_{C^m} \ge K \nu^{-m\kam} 
\Big\} \to 1\quad \text{as} \quad \nu\to0, 
\ee
for each $K>0$. 
 \end{theorem}

\begin{proof} 
{\bf Step 1}  (preliminaries): 
Consider the complement to the event in \eqref{1.0}:
$$
Q=Q^\nu =\big\{ \sup_{\tau_0 \le\tau \le\tau_0+\tau'}  | u(\tau)|_{C^m} < K \nu^{-m \kam}  \big\} .
$$
To prove  \eqref{1.0} we assume that there exists a $\gamma>0$ and a sequence $\nu_j\to0$ such that 
\be\label{1.1}
\PP (Q^{\nu_j} ) \ge 5\gamma \quad \text{for} \quad 
 \, j =1,2, \dots,
\ee
and will derive a contradiction. Below we write $Q^{\nu_j}$ as $Q$ and always suppose that 
$$
\nu \in \{\nu_1, \nu_2, \dots\}. 
$$
Without lost of generality we assume that $\tau_0=0$ and $\tau'=1$. The constants below may depend on $K$, on the norms $|u_0|_\infty$ and $\| u_0\|_m$, 
 but not on $\nu$. 

Let us denote $J_1 = [0, 1]$. Due to Theorem \ref{sup-norm}, 
$$
\PP (Q_1) \ge 1-\gamma, \qquad Q_1 = \{\sup_{ \tau\in J_1} | u(t)\Co \le C_1(\gamma) \}
$$
uniformly in  $\nu$, for a suitable $C_1(\g)$. Then, due to the Hadamard-Landau-Kolmogorov interpolation 
inequality, 
\be\label{1.2}
| u(t) \Cm \le  C_\mu(\gamma) \nu^{-\mu \kam}\quad \forall\, \omega \in Q\cap Q_1,\; \; t\in J_1, 
\ee
for any integer $\mu \in [0,m]$ (and any $\nu \in \{\nu_1, \nu_2, \dots\}$). 

Denote $J_2 =[ 0,\frac{1}{2}]$ and consider the stopping time 
\be\label{t1}
\tau_1 = \inf\{ \tau\in J_2: |u(\tau)\Cdv \ge   {K}  \nu^{-2 \kam}\}.
\ee
Then $\tau_1\leqslant \frac{1}{2}$ and $ \tau_1=\frac{1}{2}$ for  $\omega \in Q\cap Q_1$. So due to \eqref{HG} with $\Gamma= {K}  \nu^{-2 \kam}$, for any $\chi>0$ we have 
\[\begin{split}
\EE\big( \!\int_{J_2} \mathbb{I}_{[0, \chi]} (\| u(s)\|_0) ds\, \mathbb{I}_{Q\cap Q_1} (\omega) \big)
= \EE\big( \!\int_0^{\min(\frac{1}{2}, \tau_1)}  \mathbb{I}_{[0, \chi]} (\| u(s)\|_0) ds\, \mathbb{I}_{Q\cap Q_1} (\omega) \big)\\
\le  \EE  \!\int_0^{\min(\frac{1}{2}, \tau_1)}  \mathbb{I}_{[0, \chi]} (\| u(s)\|_0) ds  
\le  \EE  \!\int_0^{\min(\frac{1}{2}, \tau_\Gamma)}  \mathbb{I}_{[0, \chi]} (\| u(s)\|_0) ds \le  C {K}  \nu^{-2 \kam}\, \chi. 
\end{split}\]
Consider the event
$$
\Lambda = \{ \omega \in Q\cap Q_1: \| u(s)\|_0 \le \chi  \;\; \forall \, s\in J_2\}. 
$$
Due to the above,
$$
\PP(\Lambda ) \le  2\EE\big( \!\int_{J_2} \mathbb{I}_{[0, \chi]} (\| u(s)\|_0) ds\, \mathbb{I}_{Q\cap Q_1} (\omega) \big) \le 2 C {K}  \nu^{-2 \kam}\, \chi. 
$$
So $\PP(\Lambda ) \le \g$ if we choose 
\be\label{chi-a}
\chi = c_3(\gamma)  \nu^{2 \kam}, \qquad 
 c_3(\gamma) =  \g 
 (2C {K})^{-1} .
\ee
Consider the event
$$
Q_2 = (Q \cap Q_1) \setminus \Lambda, \qquad \PP(Q_2) \ge 3\gamma. 
$$

For $\chi$ as in \eqref{chi-a} consider the stopping time 
\begin{equation}\label{stop-l2}
\tau_2 = \inf\{ \tau \in J_1:\| u(s)\|_0 \ge \chi\}.
\end{equation}
Then $\tau_2 \leqslant\frac{1}{2}$ for all $\omega \in Q_2$. Now we consider the function 
$$
v(\tau,x) := u^\omega (\tau_2+\tau,x), \quad \tau \in J_2 =[0,\tfrac{1}{2}]. 
$$
It satisfies the equation \eqref{the_eq} with modified Wiener processes and with  initial data $v_0(x) = u^\omega(\tau_2,x)$.
\medskip

 {\bf Step 2}  (the radius-function $|v|(\tau, x)$): 
 If $\om\in Q_2$, then 
$
| v_0(\cdot)|_\infty \ge \chi. 
$
We define the stopping time $\tau_1\in J_2$ by relation \eqref{t1} with $u$ replaced by $v$. Then $\tau_1\leqslant\frac{1}{2}$ and $\tau_1 =\frac{1}{2} $ for $\om\in Q_2$.

Since $v_0(0)=0$, we can find a point $x_0\in K^n\subset \R^n$ such that $| v_0(x_0)|=\chi$. Considering the ray $R$ in $\R^n$ 
through this point, 
$R:= \R_+  x_0 \sim \R_+$,  we, firstly, find there the smallest point $x_2$ 
where  $| v_0(x_2)|=\chi$, and, secondly, find on $[0,x_2]$ the biggest point $x_1$ such that  $| v_0(x_1)|= \tfrac12\chi$. We are interested in the behaviour of $v^\om(t,x)$ for $x$ in  the segment 
$$
L=L(v_0^\om) = [x_1^\om,x_2^\om] \subset \{ x: \tfrac12 \chi \le |v_0(x)| \le \chi\}\cap K^n, 
 \qquad  L\subset R. 
$$
We will study this behaviour for $\tau$ from a time-interval $J_2 =[0, \tau_*]$ such that there  still
\be\label{still}
 \tfrac14 \chi \le |v^\om(\tau,x)\!\mid_{L}|  \le 2\chi \quad \text{for} \quad x\in L.
 \ee
 
 Applying Ito's formula to $I(\tau, x) = 
 | v( \tau\wedge \tau_1,x)|^2$ we have
 \be\label{ito_I}
 \begin{split}
 I(\tau,x) - I(0,x) +\int_0^\tau \mathbb{I}_{\{ s\le\tau_1\}} 
 \big[ -2\langle \Delta \vs, \vs\rangle +2 \sum_d b_d^2 \| \phi_d(x)|^2\big] ds\\
 =2 \int_0^\tau \mathbb{I}_{\{ s\le\tau_1\}}  \sum b_d \, \phi_d(x) \langle  \vs , d\tilde\beta_d(s)\rangle =: M(\tau,x).
 \end{split}
 \ee
 Now we extend the segment $L=[x_1,x_2] \subset R\sim \R_+$ to the segment 
 $$
 L^+ =[x_1, \max(x_2,x_1+1)] \subset \R_+, \qquad 1\le |L^+| \le \sqrt{n\pi},
 $$ 
 and 
 consider the space 
 $\cH^\rho =H^\rho(L^+)$, $1/2< \rho\le1$. We will  regard $I$ as a semimartingale  $I(\tau)\in \cH^\rho$ and $M$ -- as a martingale in $\cH^\rho$.  As
 $$
|     \phi_d(\cdot ) v(s,\cdot)  |_{\cH^\rho} \le C_\rho | \phi_d  |_{C^1} | v(s,\cdot) |_{\cH^\rho} ,
$$
and since for $s\le \tau_1$ 
$$
 | v(s,\cdot) |_{\cH^\rho} \le  \| v(s,\cdot) \|_{L_2(L^+)}^{1-\rho}  \| v(s,\cdot) \|_{H^1(L^+)}^\rho \le C\,
   | v(s,\cdot) |_{\infty}^{1-\rho}  \| v(s,\cdot) \|_{C^1(K^n)}^\rho  \le C|v(s,\cdot)|_\infty^{1-\frac{\rho}{m}} \nu^{-\rho\kappa},
$$
noting that 
\[\sup_{0\leqslant \tau\leqslant\frac{1}{2}}|v(\tau)|_\infty\leqslant \sup_{0\leqslant\tau\leqslant 1}|u(\tau)|_\infty,\]
 then in view of Theorem \ref{sup-norm} we have 
\be\label{Doob2}
\EE | M(\tau)|_{\cH^\rho} ^2 
\le C_{\rho } \int_0^\tau  \nu^{-2\rho\kappa}  \sum b_d^2\, |d|^2ds
 =  C_{\rho} B_{1} \nu^{-2\rho\kappa} \tau. 
\ee
So by Doob's inequality,
$$
\PP\left( \sup_{0\le \tau \le\tau_*}  | M(\tau)|_{\cH^\rho} ^2\ge r^2\right) 
\le { C'_{\rho} \nu^{-2 \rho\kappa} \tau_*}\, {r^{-2}} .
$$
Since $\cH \subset C^0(L^+)$, then 
\be\label{Q3}
\PP(Q_3) \ge1-\gamma, \qquad Q_3= \{  \sup_{0\le \tau \le\tau_*}  | M(\tau)|_{C^0(L^+) } \le C^3(\g) \nu^{-\rho\kappa} \sqrt{\tau_*} \,\},
 \ee
 for a suitable $C^3(\g)$ (depending on $\rho$). 
 
Now let us choose 
 \be\label{tau_st}
\tau_* = c^4(\g) \nu^{2(4+\rho)\kappa}<\frac{1}{2}. 
 \ee
 Then from \eqref{ito_I} and \eqref{1.2}, 
 \be\label{I_est} 
 | I(\tau,x) - I(0,x) |  \le C_\g (\tau_* \nu^{-2\kam} + \nu^{-\rho\kam}  \sqrt{\tau_*} ) \le \tfrac1{8} \chi^2, \quad \tau\in J_3:=[0, \tau_*], \; x\in L, 
  \ee
  for $\om \in Q_2\cap Q_3$, 
if $c^4(\g) \ll 1$.  This implies \eqref{still} for $\om\in Q_2\cap Q_3$ with $\tau_*$ as in \eqref{tau_st}, since  $|v_0| \ge \tfrac12 \chi$ on $L$.
\medskip

 {\bf Step 3}  (the  angles):  Finally 
 we examine the behaviour of the angles Arg$\,v^\om(\tau,x)$ for $x$ in $L(v_0)$. To do this we consider the angle-function on the   annulus 
$$
Ann = \{z: \tfrac14 \chi \le |z| \le 2 \chi\},
$$
i.e.
$$
\Phi: Ann \mapsto S^1, \quad z \mapsto \text{Arg} \,z,
$$
and define the stopping time 
$\ 
\tau_x = \inf\{ \tau \ge 0: v^\om(\tau,x) \in \overline{\C\setminus  Ann }\}.
$
Then $\tau_x=0$  if $x=0$ and $\tau_x \ge \tau_*$ if $x\in L(v_0)$ and $\om \in Q_2\cap Q_3$. For a fixed ``past" $v_0$ and $x \in L(v_0)$ let us consider the
random process
$
\phi_x^\om(\tau) =$ Arg$\, v^\om(\tau\wedge  \tau_x,x).
$
Applying  to it Ito's formula, we get:
  \begin{equation}\label{ito-phi}
 \begin{split}
 \phi_x&(\tau) -  \phi_x(0) + \nu^{-1}\int_0^{\tau\wedge \tau_x} d\Phi (\vs) (i|v|^2\vs)ds\\
  =&\int_0^{\tau\wedge \tau_x} \Big(d\Phi (\vs) (\Delta \vs ) 
   +  \sum_d b_d^2 d^2\Phi(\vs) (\phi_d(x), \phi_d(x))\Big)ds\\
   + &\int_0^{\tau\wedge \tau_x} \sum_d b_d 
   d\Phi(\vs)  (\phi_d(x) d\tilde\beta_d(s)), \quad \tau\ge0,\;\;  x \in L(v_0).
 \end{split}
  \end{equation}
 Let us denote the stochastic integral in the r.h.s. as $N_x(\tau)$. It is convenient to regard $\phi_x(\tau)$ as a point in the real line rather than in $S^1$. To
 do that, if $\tau_x=0$ for some $x\in L$, 
 we take $\phi_x^\om(0)  \in [0,2\pi)$. Otherwise 
   we take for  $\phi_{x_1}^\om(0) $ the value of Arg$\, v^\om_0(x) \in [0,2\pi)$, continuously  extend it to a function $\phi_{x}^\om(0) $  
   on $L$, and then 
  construct  $\phi_x^\om(\tau) $  from \eqref{ito-phi} by continuity. 
 
Since
$$
|d \Phi(z)| \le C\chi^{-1}, \quad |d^2 \Phi(z)| \le C\chi^{-2} \quad \forall\, z\in Ann,
$$
then for $\om\in Q_2\cap Q_3$ and for $\tau \le\tau_*$ the sum of the two deterministic integrals in the r.h.s. of \eqref{ito-phi}  is bounded by 
$$
C_\g \tau_*\big[ \chi^{-1} \nu^{-2\kam} + \chi^{-2}\big] \le C_\g \nu^{(4+2\rho)\kappa}.
$$

Now consider the stochastic integral   $N_x(\tau)$. Since 
$
\EE | d\Phi(\vs) (\phi_d)|^2 \le C\chi^{-2},
$
then 
$
\EE | N_x(\tau)|^2 \le C_\gamma \chi^{-2} \tau. 
$
So for any $x\in L$, 
$$
\PP( Q_4^x) \ge 1- \g/2, \quad Q_4^x = \{ \sup_{0\le \tau \le\tau_*} |N_x(\tau)| \le C^4_\g \chi^{-1}  \sqrt{\tau_*} \},
$$
for a suitable $C^4_\g$.  Let us define 
$
Q_4 =Q_4^{x_1} \cap Q_4^{x_2} 
$
and consider 
$$
\hat Q = Q_2 \cap Q_3\cap Q_4, \qquad \PP \hat Q \ge \g. 
$$
For any $\om\in \hat Q$ and  $j=1,2$ 
 we have  $\tau_{x_j} \ge \tau_*$,  and 
$$
d\Phi (v(s,x_j))( i|v|^2 \vs) =| v(s,x_j)|^2,
$$
for $0\le s \le \tau_*$. Due to \eqref{I_est}, 
$
| v(s,x_1)|^2 \ge \frac{7}{8}\chi^2 $ and $  | v(s,x_2)|^2 \le \frac{5}{8}\,  \chi.
$
Therefore in view of  the Ito's formula for $\phi_x$,  for any $\om\in\hat Q$,  
\[\begin{split}
\phi_{x_1}(\tau_*)& \ge \nu^{-1} \tau_* \tfrac78\chi^2\, - \, C_\g( \nu^{(4+2\rho)\kam} + \chi^{-1}\sqrt{\tau_*} )\\
&\ge  \tfrac78 c^4(\g) C_3(\g) 
 \nu^{-1 +\kam (12+2\rho )} - C_\g (\nu^{(4+2\rho)\kam}  + \nu^{\kam(2+\rho) } ); 
 \\
\phi_{x_2}(\tau_*)& \le
  \tfrac58 c^4(\g) C_3(\g)  \nu^{-1 +\kam (12+2\rho )}
   +C_\g \nu^{\rho\kam}.
\end{split}\]
Since $\kam<1/16$, then choosing $\rho$ close to $1/2$  we achieve that 
\be\label{lastest}
\phi_{x_1}(\tau_*) - \phi_{x_2}(\tau_*) \ge C(\g)  \nu^{-1 +\kam (12+2\rho)} , 
\ee
if $\nu\ll1$. From the other hand, since for $\om\in \hat Q$,
$
|v(\tau_*,x)\Cod  \le K \nu^{-\kam} 
$
and 
$| v(\tau_*,\cdot)|$ 
 is
$\ge \tfrac14 \chi$ on the segment $L$, then 
$$
| \nabla \text{Arg}\, v(\tau_*,\cdot) \!\mid_{L} | \le 4K \nu^{-\kam} \chi^{-1} = C(\g) \nu^{-3\kam}. 
$$
So for any $\om\in \hat Q$ we must have 
$
|\phi_{x_1} (\tau_*) - \phi_{x_2} (\tau_*) | \le C(\gamma) \nu^{-3\kam}.
$
Combining here and \eqref{lastest}, we obtain
$$
\nu^{-1 +\kam(15+2\rho)} \le C(\g). 
$$
As for $\rho$ we can take any number $>1/2$ and the last  relation  holds  for arbitrarily small $\nu$, then
$
\kam \ge \frac1{16}. 
$
This conclusion has been obtained for any $\om$ from the event 
$\hat Q$, where  $\PP\hat Q>\g$. The obtained contradiction with the assumption of the theorem proves the assertion \eqref{1.0}.  
\end{proof}

\section*{Acknowledgment}
GH is supported by NSFC (Significant project No.11790273) in China and SK  thanks the {\it Russian Science Foundation\/} for support through the grant  18-11-00032.

\bibliography{reference}{}
\bibliographystyle{plain}

\end{document}